\documentclass[11pt, letterpaper, reqno]{amsart}
\usepackage[utf8]{inputenc} 	
\usepackage{microtype} 			
\usepackage{geometry}
\usepackage{amsmath}			
\usepackage{amsthm}		 		
\usepackage{amssymb}	 		 			
\usepackage{bm}					
\usepackage{mathrsfs}			
\usepackage{thmtools}
\usepackage{xcolor}				
\usepackage{graphicx}
\usepackage[all]{xy}			
\usepackage{enumitem} 			
\usepackage{array}
\usepackage{lmodern}
\usepackage[T1]{fontenc}		
\usepackage[colorlinks=true, citecolor=cyan, urlcolor=magenta]{hyperref}
\usepackage{cleveref}
\usepackage[backend=biber, style=alphabetic, backref=true, firstinits=true, isbn=false]{biblatex}
\makeatletter
\renewbibmacro{in:}{}
\DeclareFieldFormat{pages}{#1}
\renewcommand*{\bibnamedash}{%
	\leavevmode\raise +0.6ex\hbox to 5.5ex{\hrulefill}.\space\space}

\InitializeBibliographyStyle{\global\undef\bbx@lasthash}

\newbibmacro*{bbx:savehash}{\savefield{fullhash}{\bbx@lasthash}}

\renewbibmacro*{author}{%
	\ifboolexpr{
		test \ifuseauthor
		and
		not test {\ifnameundef{author}}
	}
	{%
		\iffieldequals{fullhash}{\bbx@lasthash}
		{\bibnamedash\addcomma\space}
		{\printnames{author}}%
		\usebibmacro{bbx:savehash}%
		\iffieldundef{authortype}
		{}
		{%
			\setunit{\addcomma\space}%
			\usebibmacro{authorstrg}%
		}%
	}
	{\global\undef\bbx@lasthash}%
}
\makeatother
\newtheorem{propositionx}{Proposition}[section]
\newenvironment{proposition}
{\pushQED{\qed}\propositionx}
{\popQED\endpropositionx}
\newenvironment{propositionp}
{\pushQED{\qed}\propositionx}
{\popQED\endpropositionx}
\newtheorem{theoremx}[propositionx]{Theorem}
\newenvironment{theorem}
{\pushQED{\qed}\theoremx}
{\popQED\endtheoremx}
\newenvironment{theoremp}
{\pushQED{\qed}\theoremx}
{\popQED\endtheoremx}
\newtheorem{corollaryx}[propositionx]{Corollary}
\newenvironment{corollary}
{\pushQED{\qed}\corollaryx}
{\popQED\endcorollaryx}
\newtheorem{lemmax}[propositionx]{Lemma}
\newenvironment{lemma}
{\pushQED{\qed}\lemmax}
{\popQED\endlemmax}

\theoremstyle{remark}
\newtheorem*{remark}{Remark}

\newcommand{\bbB}{\mathbb{B}}
\newcommand{\bbC}{\mathbb{C}}

\newcommand{\bbE}{\mathbb{E}}

\newcommand{\bbN}{\mathbb{N}}

\newcommand{\bbP}{\mathbb{P}}
\newcommand{\bbQ}{\mathbb{Q}}
\newcommand{\bbR}{\mathbb{R}}
\newcommand{\bbS}{\mathbb{S}}
\newcommand{\bbT}{\mathbb{T}}

\newcommand{\bbZ}{\mathbb{Z}}

\newcommand{\calF}{\mathcal{F}}

\newcommand{\calI}{\mathcal{I}}

\newcommand{\calU}{\mathcal{U}}

\newcommand{\scrC}{\mathscr{C}}
\newcommand{\scrD}{\mathscr{D}}

\newcommand{\scrH}{\mathscr{H}}

\newcommand{\scrL}{\mathscr{L}}

\newcommand{\scrS}{\mathscr{S}}

\newcommand{\scrX}{\mathscr{X}}
\newcommand{\scrY}{\mathscr{Y}}

\newcommand{\bfa}{\mathbf{a}}
\newcommand{\bfb}{\mathbf{b}}

\newcommand{\bmdelta}{\bm{\delta}}

\newcommand{\bmmu}{\bm{\mu}}

\newcommand{\bmsigma}{\bm{\sigma}}

\newcommand{\dd}{\,\mathrm{d}}

\title{The microlocal irregularity of Gaussian noise}
\author{Ethan Sussman}
\email{ethanws@mit.edu}
\address{Department of Mathematics, Massachusetts Institute of Technology, Massachusetts 02139-4307, USA}
\date{November 10th, 2022 (Last Update), April 25th, 2022 (Published), December 13th, 2020 (Original)}

\begin{document}

\begin{abstract}
	The study of random Fourier series, linear combinations of trigonometric functions whose coefficients are independent (in our case Gaussian) random variables with polynomially bounded means and standard deviations, dates back to Norbert Wiener in one of the original constructions of Brownian motion \cite{Wiener}.
	A geometric generalization -- relevant e.g.\ to Euclidean quantum field theory with an infrared cutoff -- is the study of random Gaussian linear combinations of the eigenfunctions of the Laplace-Beltrami operator on an arbitrary compact Riemannian manifold $(M,g)$,  \textit{Gaussian noise} $\Phi$.

	I will prove that, when our random coefficients are independent Gaussians whose standard deviations obey polynomial asymptotics and whose means obey a corresponding polynomial upper bound, the resultant random $\scrH^s$-wavefront set $\operatorname{WF}^s(\Phi)$ (defined as a subset of the cosphere bundle $\bbS^*M$) is either almost surely empty  or almost surely the entirety of $\bbS^*M$, depending on $s \in \bbR$, and we will compute the threshold $s$ and the behavior of the wavefront set at it. Consequently, the random $C^\infty$-wavefront set $\operatorname{WF}(\Phi)$ is almost surely the entirety of the cosphere bundle. 
	The method of proof is as follows: using Sazonov's theorem and its converse, it suffices to understand which compositions of microlocal cutoffs and inclusions of $L^2$-based fractional order Sobolev spaces are Hilbert-Schmidt (HS), and the answer follows from general facts about the HS-norms of the elements of the pseudodifferential calculus of Kohn and Nirenberg. 
\end{abstract}

\maketitle

\tableofcontents

\section{Introduction}
\label{sec:introduction}

This paper gives a microlocal application of an approach  due largely to the participants of the \emph{S\'eminaire Maurey-Schwartz} \cite{SeminaireSchwartz69}, Laurent Schwartz in particular, to understanding the pathwise regularity of certain random distributions -- e.g.\ Brownian motion (a.k.a.\ the Wiener process \cite{Wiener}), the Ornstein-Uhlenbeck process, the Gaussian free field, white noise in arbitrarily many dimensions, and sufficiently small perturbations thereof. This approach
\begin{enumerate}
	\item blackboxes all of the probabilistic and measure-theoretic reasoning into Sazonov's theorem \cite{Sazonov} (or its variants, generalizations, corollaries, and converses), and then
	\item proceeds \emph{non}-probabilistically by way of classical functional analytic estimates, e.g.\ the various Sobolev embedding theorems and their various refinements. 
	\label{item:2}
\end{enumerate}
The study of the regularity of random distributions began with the $\alpha$-H\"older regularity of the Wiener process, initiated by Wiener himself. 
We consider here Sobolev regularity, which in the 1D case seems to have first been considered by Kushner \cite{Kushner}.

Partial motivation comes from constructive quantum field theory. Nelson \cite{Nelson1}\cite{Nelson2} constructed the scalar field in Euclidean spacetime, amounting to a Gaussian measure
\begin{equation} 
\Gamma_m :\operatorname{Borel}(\scrS'(\bbR^d)) \to [0,1],
\label{eq:2}
\end{equation} 
($m>0$ is the ``mass'' of the field) 
on the conuclear space $\scrS'(\bbR^d) = \scrS(\bbR^d)^*$ of tempered distributions. See \Cref{figures} for samples. Physicists typically attempt to define interacting quantum field theories via so-called path-``integrals'' --- these are ill-defined (except in some formal perturbation-theoretic sense), and attempts to mathematically rehabilitate their approach beyond the level of formal perturbation theory run into difficulty because Nelson's measure is only supported on low regularity Sobolev spaces. 
In the $d=1$ case, the regularity is not too low for the construction to make sense.  The result is the \emph{Feynman-Kac formula} \cite{SimonFunctionalIntegration}. In the $d\geq 2$ case, the singularities are more severe, and the resultant difficulties are more serious. They can, in some cases, be circumvented, but only after rather serious work \cite{JGPhi43}\cite{GJorigI}\cite{GJorigII}\cite{GJorigIII}\cite{GJorigIV} (on the Lorentzian side) and \cite{Feldman}\cite{Guerra}\cite{Sokal} (as an obviously incomplete sample of the Euclidean side) --- see \cite{SimonQFT}\cite{GlimmJaffe}\cite[\S24]{SimonFunctionalIntegration} for a summary and further references. 
The recent and extensive work of Martin Hairer \cite{HairerTheory}\cite{HairerTheory2} and collaborators  -- e.g.\ \cite{HairerCollab2}\cite{HairerCollab1}\cite{HairerCollab3}\cite{HairerCollab5}\cite{HairerCollab6}\cite{HairerCollab4} -- on the dynamical $\Phi^4_2,\Phi^4_3$ models and variants can be considered an outgrowth of these earlier investigations, albeit from the perspective of SPDE.
One of the earliest results on the $L^2$-based Sobolev regularity of the free field is \cite{Cannon74}. See also \cite{ColellaLanford}\cite{Reed74}. 

A special case of the main theorem of this paper, stated using some standard notation which is recalled in \S\ref{sec:setup}, is:
\begin{theorem}
	\label{thm:prelim}
	Let $(M,g)$ denote a compact Riemannian manifold, and let $\Gamma : \operatorname{Borel}(\scrD'(M)) \to [0,1]$ denote the Gaussian measure with covariance $(1+\triangle_g)^{-1}:\scrD(M)\to\scrD(M)$. (See \cite[Chp. 6]{GlimmJaffe}.) Then, 
	\begin{equation}
	\operatorname{WF}^s(B) = 
	\begin{cases}
	\varnothing & (s<1-d/2) \\
	\bbS^* M & (s\geq 1-d/2) \\
	\end{cases}
	\label{eq:mm}
	\end{equation}
	for $\Gamma$-almost all $B \in \scrD'(M)$. 
\end{theorem}
Here $\triangle_g$ is the positive semidefinite (a.k.a.\ ``geometer's'') Laplacian. Given a (not necessarily complete) probability space $(\Omega,\calF,\bbP)$, we say that some subset $A\subseteq \Omega$ occurs ``($\bbP$-)almost surely'' if there exists some set $F\in \calF$ such that $A\supseteq F$ and $\bbP(F) = 1$. In many of the cases relevant to this paper, we can take $A=F$, and this will be sufficiently clear in context that we refrain from pointing it out.
\begin{remark}
	We stated \Cref{thm:prelim} for the massive Gaussian free field, but the framework below applies to the massless free field \cite{GFF1}\cite{GFF2} (i.e.\ the Gaussian noise with ``covariance $\triangle_g^{-1}$'') as well. This is Gaussian noise which, in the language below, has
	\begin{equation} 
	\sigma_n = 
	\begin{cases}
		0 & (n=0)\\ 
		\lambda_n^{-1/2} & (n>0). 
	\end{cases}
	\end{equation}
	(Cf.\ \cite[Eq. 1]{GFF1}.)
	Consequently, \Cref{thm:main} applies to the massless case, with the same conclusion (since the leading order asymptotics of $\smash{\lambda_n^{-1}}$ and $\smash{(1+\lambda_n)^{-1}}$ for $n\to\infty$ agree, $\lambda_n$ being the $n$th eigenvalue of the Laplace-Beltrami operator). 
\end{remark}

\begin{remark}
	\Cref{thm:prelim} constructs (in a probabilistic sense) many distributions on $M$ whose wavefront sets are as large as possible. The construction of distributions with prescribed wavefront set $C\subset \bbS^* M$ (where $C$ is closed) can be found in \cite[Theorem 8.1.4]{HormanderI}, which includes the case $C=\bbS^* M$. 	
\end{remark}

The global version of \cref{eq:mm} can be found in Schwartz's treatise, \cite[Part II-Chapter V, \S3, pg.\ 280]{SchwartzRadon}, and our argument in  \S\ref{sec:main_theorem} is  essentially a microlocal refinement of his.

The connection to Euclidean QFT is as follows: even if the massive Gaussian free field were almost surely irregular somewhere, we could, following H\"ormander \cite[Chapter 8]{HormanderI}, still make direct sense of its powers if it were the case that the wavefront set $\operatorname{WF}(B)$ were almost surely one-sided, meaning that 
\begin{equation}
	(x,\hat{\xi}) \in \operatorname{WF}(B) \subset \bbS^* M \Rightarrow (x,-\hat{\xi}) \not\in \operatorname{WF}(B).
	\label{eq:osd}
\end{equation}
for $\Gamma$-almost all $B\in \scrD'(M)$. 
A reality argument suffices to show that this is not the case for $\Gamma$, but it is not obvious for arbitrary noises. 
\Cref{thm:prelim} proves that \cref{eq:osd} is almost surely as false as possible. 
In contrast, the Schwinger functions of Euclidean QFTs have or are expected to have constrained wavefront sets, and the Lorentzian version of this is an axiom of perturbative QFT on curved spacetimes --- see  \cite{Iagolnitzer}\cite{Radzikowski}\cite{Verch}\cite{Brouder:2014hta}\cite{Fredenhagen} and the references therein.

\Cref{thm:prelim} should be contrasted with quantum ergodicity \cite{Zelditch}. Quantum ergodicity applies on Riemannian manifolds whose geodesic flows are classically chaotic, but our main theorem applies under no restrictions  on $(M,g)$ whatsoever. The fact that random objects constructed out of the eigenfunctions of the Laplacian tend to have better (microlocal) equidistribution properties than deterministic ones has been known since at least \cite{Zelditch92}. The literature on quantum ergodicity is vast, and we cannot give an adequate summary here.

\begin{remark}
	It is not difficult to see that \cref{eq:mm} defines a Borel subset of $\scrD'(M)$, so that $\Gamma(\{B \in \scrD'(M):\text{\cref{eq:mm} holds}\})$ is well-defined for each $s\in \bbR$, and indeed, this is implicit in the discussion below. Moreover, $\operatorname{Borel}(\scrD'(M))$ is the $\sigma$-algebra generated by evaluation against an eigenbasis of the Laplace-Beltrami operator --- cf.\ \Cref{lem:measurability_lemma}. 
	\emph{A priori}, after writing $B$ as a random linear combination of the elements of the given eigenbasis (all of which are smooth), the Kolmogorov zero-one law tells us that
	\begin{equation}
	\Gamma(\{B\in \scrD'(M) :\text{\cref{eq:mm} holds}\}) \in \{0,1\}
	\label{eq:nn}
	\end{equation}
	for each $s$.
	\Cref{thm:prelim} tells us which of these two possibilities holds. The measure theoretic technicalities involved here will be discussed more in \S\ref{subsec:random}. 
\end{remark} 

\begin{corollary}
	\label{cor:Holder}
	Let $(M,g)$ denote a compact Riemannian manifold, and let $\Gamma:\operatorname{Borel}(\scrD'(M))\to [0,1]$ denote the Gaussian measure with covariance $(1+\triangle_g)^{-1} : \scrD(M)\to \scrD(M)$. Then, if $d\geq 2$, 
	\begin{equation}
	\chi  \in \scrC^\infty(M), \chi B \in \scrC^0(M) \Rightarrow \chi = 0
	\end{equation}
	for $\Gamma$-almost all $B\in \scrD'(M)$. That is, almost surely $B$ is nowhere locally equal to a continuous function. 
\end{corollary} 
\begin{remark} 
This corollary is expected to be sharp in the sense that, if $d=2$, $(1+\triangle_g)^{-\epsilon} B \in \scrC^0(M)$ for all $\epsilon>0$ for $\Gamma$-almost all $B\in \scrD'(M)$. This is known to hold for $M=\bbT^2$, but we do not pursue the case of general $M$ here. It should follow from $L^p$-based analogues of \Cref{thm:prelim} in conjunction with the Sobolev embedding theorems. 
\end{remark}

\begin{figure}[p]
	\centering
	\includegraphics[width = .7\textwidth]{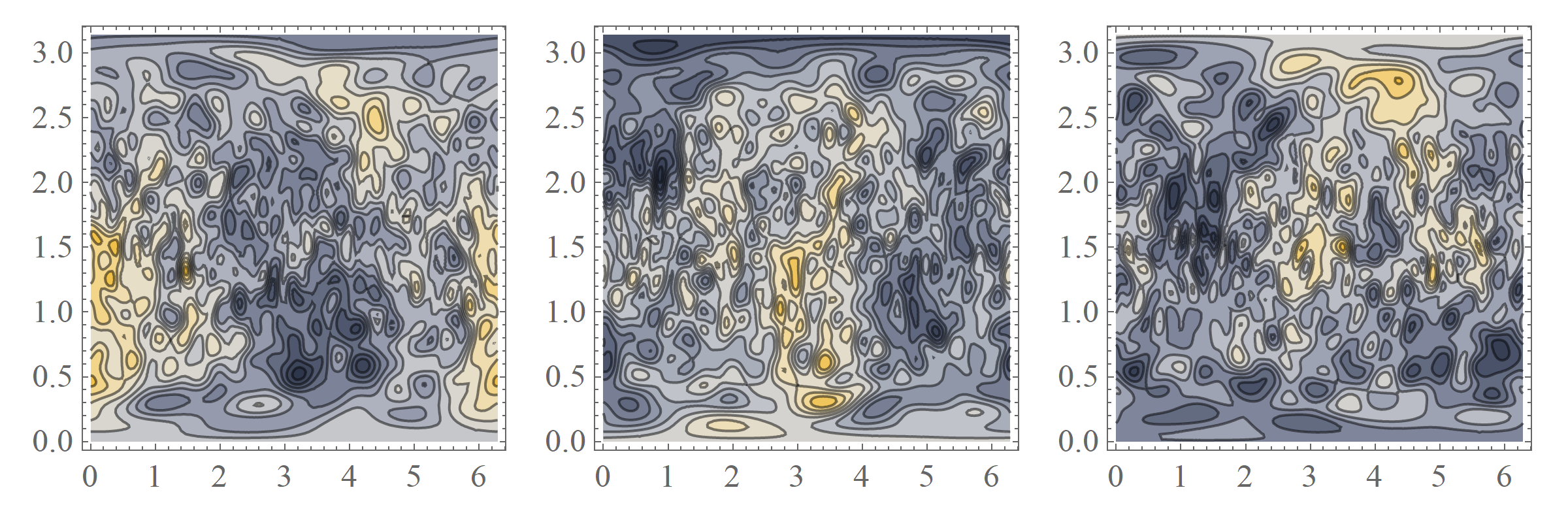}
	\includegraphics[width = .7\textwidth]{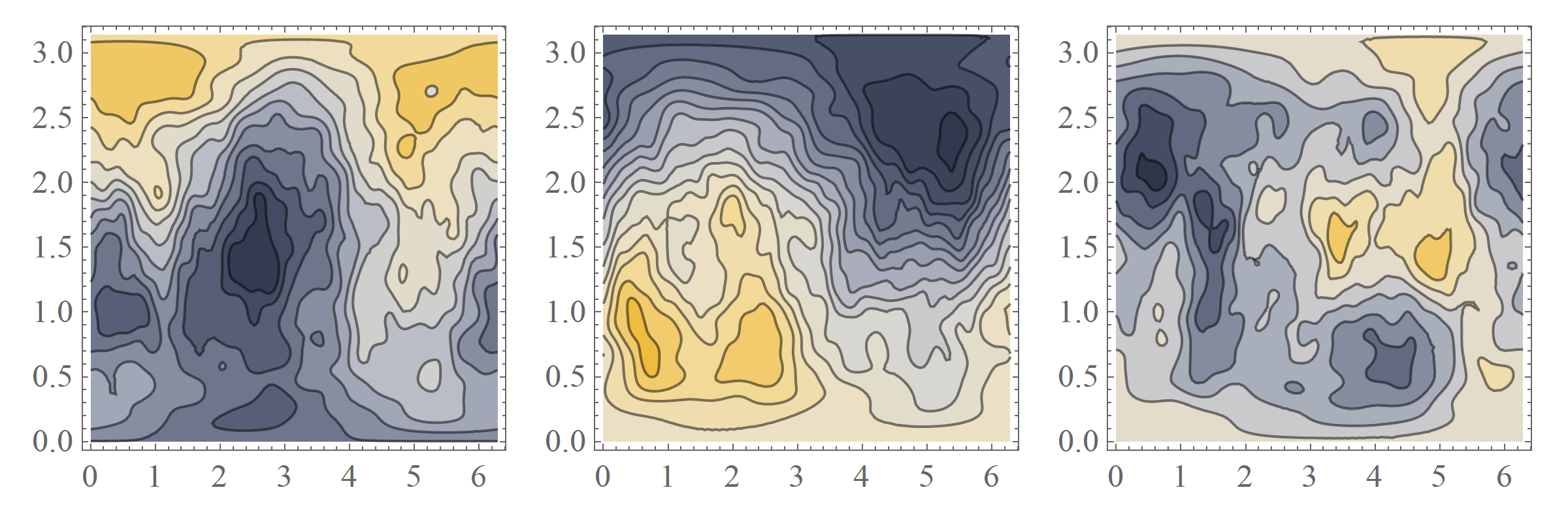}
	\includegraphics[width = .7\textwidth]{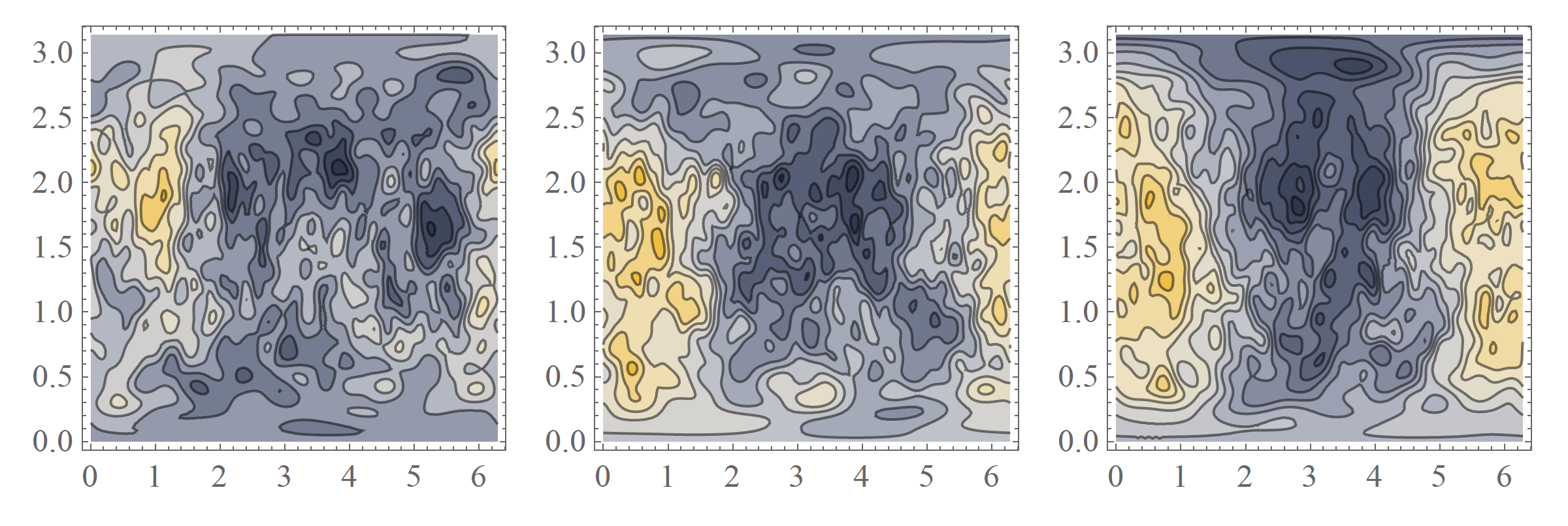}
	\includegraphics[width = .7\textwidth]{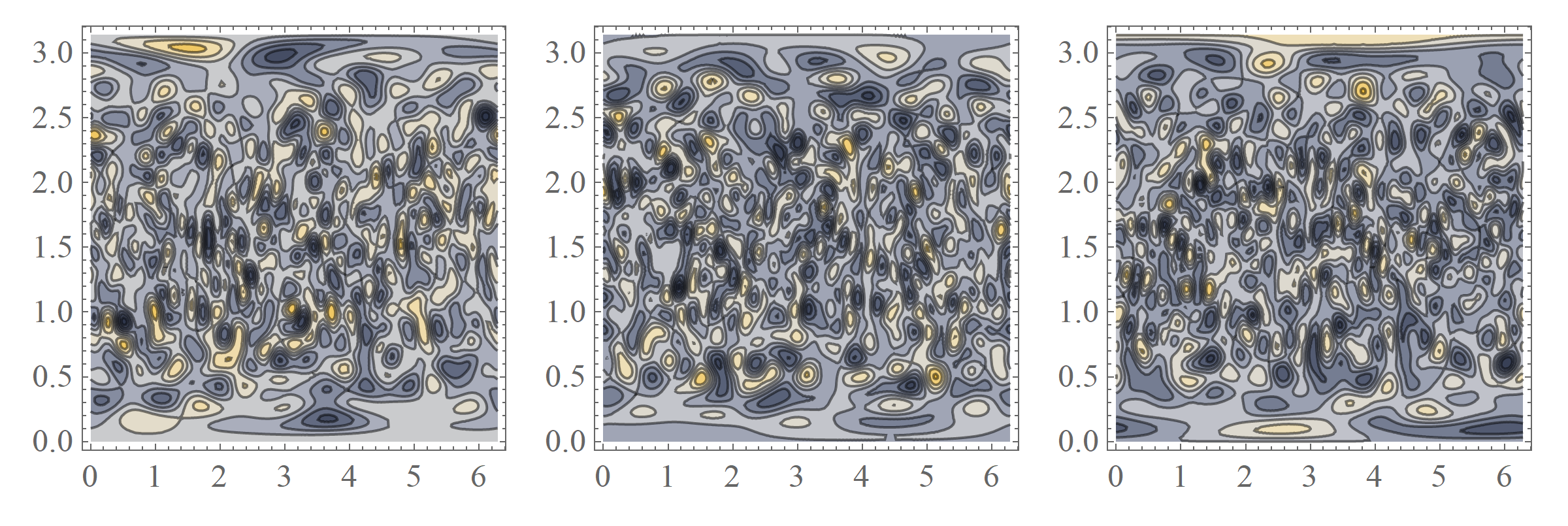}
	\caption{\textit{Top Row}: Three samples $\Phi_{40}(\omega_1),\Phi_{40}(\omega_2),\Phi_{40}(\omega_3)$ of (partial sums of) Gaussian noise on the 2-sphere $\bbS^2$ with damping $\sigma_n \sim 1/n^{1/2}$ (i.e.\ the Euclidean free field \cite{Nelson2}, on the 2-sphere).  By \Cref{thm:prelim}, $\Phi_N$ just barely fails to converge a.s.\ in $L^2(\bbS^2)$ as $N\to\infty$ but converges in all $L^2$-based Sobolev spaces of slightly negative regularity. (Plotted using an equirectangular projection.) \textit{Second Row}: Four samples $\Phi_{40}(\omega_1),\cdots,\Phi_{40}(\omega_4)$ of partial sums of Gaussian noise on the 2-sphere $\bbS^2$ with damping $\sigma_n \sim 1/n$.  By \Cref{global_regularity}, $\Phi_N$ converges in $L^2(\bbS^2)$ almost surely as $N\to\infty$. \textit{Third Row}: Three samples $\Phi_{40}(\omega_1),\Phi_{40}(\omega_2),\Phi_{40}(\omega_3)$ of partial sums of Gaussian noise on the 2-sphere $\bbS^2$ with damping $\sigma_n \sim 1/n^{2/3}$. By \Cref{global_regularity}, $\Phi_N$ converges in $L^2(\bbS^2)$ almost surely. \emph{Bottom Row}: Three samples $\Phi_{40}(\omega_1),\Phi_{40}(\omega_2),\Phi_{40}(\omega_3)$ of partial sums of $L^2(\bbS^2)$-based white noise on the 2-sphere $\bbS^2$, i.e.\ Gaussian noise with damping $\sigma_n=1$. This is the spherical analogue of the noise considered in \cite{Veraar11}. } 
	\label{figures}
\end{figure} 

Both \Cref{thm:prelim} and \Cref{cor:Holder} are not obvious from the construction of $B$ as a random series, i.e.\ the Karhunen-Lo\`eve expansion \cite{Alexanderian},  since our choice of eigenbasis $\{\phi_n\}_{n=0}^\infty\subseteq \scrL^2(M)$, $\triangle_g \phi_n = \lambda_n \phi_n$, might be inhomogeneous and anisotropic. That is to say, the $\phi_n$ are global objects --- they depend on the global geometry of $(M,g)$.

Although it is not central to our exposition (since we \emph{define} Gaussian noises to be particular random linear combinations of eigenfunctions of the Laplacian), it is worth pointing out that if we define $\Gamma$ using a covariance (as above), e.g.\ via the Bochner-Minlos theorem \cite[Part II-Chapter IV]{SchwartzRadon}\cite{SimonFunctionalIntegration}\cite{GlimmJaffe} or some other method, then one can prove that the resultant measure is the law of a random series of eigenfunctions. This can be done by proving that the random series converges in an appropriate function space, computing out the Fourier transform of the random series (which is easy for Gaussian noises), and checking that it agrees with the prescribed covariance. (The key fact here is that Fourier transforms of Radon measures on separable Banach spaces uniquely determine them, and likewise for conuclear spaces \cite[Part II-Chapter IV]{SchwartzRadon}\cite[Appendix E, E.1.16 \& E.1.17]{Hytonen2016analysis}). 

Let's now consider the special case of the circle $M=\bbS^1$, where there exists a preferred choice of eigenbasis, the trigonometric functions. 
In this case there is a simple symmetry-based argument to be found in \cite[Part I, Chp.\ 9, \S7]{KahaneFourierSeries} which shows  that the local regularity of a random formal Fourier series with symmetric and independent  coefficients -- and the Brownian bridge in particular -- does not differ from the global regularity. 
The argument is general enough to apply to both Sobolev and H\"older regularity. The details of this application are left to the reader.
It is easy to see that this local result implies the full microlocal result \cref{eq:mm} for $M=\bbS^1$.  Let $\bbC((e^{i\theta}))$ denote the $\bbC$-vector space of formal Fourier series on  $\bbS^1=\bbR_\theta/2\pi \bbZ$. We imagine for each open, connected, and nonempty $I\subseteq \bbS^1$ -- a.k.a. ``(open) interval'' -- a ``property'' $\mathtt{P}(I)$, which we'll think of as a subset $\mathtt{P}(I)\subseteq \bbC((e^{i\theta}))$, which a given formal Fourier series may or may not satisfy, and we suppose that 
\begin{enumerate}
	\item $\mathtt{P}(I)$ is a $\bbC$-subspace of $\bbC((e^{i\theta}))$ containing all trigonometric polynomials,
	\begin{equation} 
	\{a_n\}_{n\in \bbZ} , \exists N \in \bbN \text{ s.t. } a_n = 0\text{ for }|n|\geq N \Rightarrow \sum_{n\in \bbZ} a_n e^{in\theta} \in \mathtt{P}(I),
	\end{equation}
	\label{item:Kahane_i}
	\item 
	if $I_0,I_1,I_2\subseteq \bbS^1$ are intervals with $I_0\subseteq I_1 \cup I_2$, then 
	\begin{equation} 
	\texttt{P}(I_0) \supseteq \texttt{P}(I_1) \cap \texttt{P}(I_2),
	\end{equation} 
	holds,
	\label{item:Kahane_ii} 
	\item if  $\Phi(\theta) \in \bbC((e^{i\theta}))$ satisfies $\mathtt{P}(I)$, then $\Phi(\theta+\theta_0)$ satisfies $\mathtt{P}(I+\theta_0)$,  
	where the expressions `$\Phi(\theta+\theta_0)$' and `$I+\theta_0$' have the obvious meanings, 
	\begin{equation}
	\Phi(\theta) = \sum_{n\in \bbZ} a_n e^{in\theta} \Rightarrow \Phi(\theta+\theta_0) = \sum_{n\in \bbZ} (a_n e^{in\theta_0}) e^{in\theta}, 
	\end{equation}
	
	\label{item:Kahane_iii}  
	\item if $\Phi(\theta) \in \bbC((e^{i\theta}))$ satisfies $\mathtt{P}(I)$, then the two formal  series $e^{\pm i \theta} \Phi(\theta)$ satisfy $\mathtt{P}(I)$ as well, where `$e^{+i\theta} \Phi(\theta)$' and `$e^{-i\theta}\Phi(\theta)$' have the obvious meanings, 
	\begin{equation}
	\Phi(\theta) = \sum_{n\in \bbZ} a_n e^{in\theta} \Rightarrow e^{\pm i \theta} \Phi(\theta) = \sum_{n\in \bbZ} a_{n\mp 1} e^{in\theta},
	\label{eq:i}
	\end{equation}
	and lastly
	\label{item:Kahane_iv}  
	\item $\texttt{P}(I)$ is a Borel subset of $\bbC((e^{i\theta}))$ for each interval $I\subseteq \bbS^1$, where $\bbC((e^{i\theta}))$ is endowed with the product topology corresponding to the canonical isomorphism 
	\begin{equation} 
	\bbC((e^{i\theta})) \cong \prod_{n\in \bbZ} \bbC, \quad \sum_{n\in \bbZ} a_n e^{in\theta} \mapsto \{a_n\}_{n\in \bbZ}.
	\end{equation}  
	\label{item:Kahane_v}
\end{enumerate}
Certainly, most -- if not all -- reasonable local notions of regularity  satisfy these five conditions. Note that they are tailored to the circle (and flat tori):
while (\ref{item:Kahane_i}), \cref{item:Kahane_ii}), and \cref{item:Kahane_v}) could be made sense of were we to replace $\bbS^1$ by an arbitrary Riemannian manifold -- namely by replacing the trigonometric functions by an arbitrary orthonormal eigenbasis of the Laplace-Beltrami operator -- items \cref{item:Kahane_iii}) and \cref{item:Kahane_iv}) would then have no clear analogue. Item \cref{item:Kahane_v}) can be rephrased slightly by noting that the Borel $\sigma$-algebra of $\bbC((e^{i\theta}))$ is given by 
\begin{equation}
\operatorname{Borel}(\bbC((e^{i\theta}))) = \sigma(a_n : n\in \bbZ) , 
\label{eq:v}
\end{equation}
where `$a_n$' is shorthand for the map $\sum_{n\in \bbZ} a_n e^{in\theta} \mapsto a_n$. The notation on the right-hand side of \cref{eq:v}) denotes the $\sigma$-algebra generated by the given countable set of functions. 

A \emph{random} formal Fourier series consists of the following data: a probability space $(\Omega,\calF,\bbP)$ and random  variables $a_n : \Omega\to \bbC$ for $n\in \bbZ$, yielding a function $\Phi(-):\Omega\to \bbC((e^{i\theta}))$, 
\begin{equation}
\Phi(\omega) = \Phi(\omega)(\theta)= \sum_{n\in \bbZ} a_n(\omega) e^{in\theta}. 
\label{eq:rfs}
\end{equation}
 $a_n$ is called \emph{symmetric} if it is equidistributed with $-a_n$. Note that $\Phi$ is measurable with respect to \cref{eq:v}, so by (\ref{item:Kahane_v}), $\bbP(\Phi^{-1}(\mathtt{P}(I))) \in [0,1]$ is well-defined for every interval $I \subseteq \bbS^1$.  
\begin{proposition}
	\label{prop:Kahane}
	Given a random formal Fourier series $\Phi$ with independent and symmetric coefficients, if with positive probability $\Phi \in \mathtt{P}(I)$ for some random interval $I=I(\omega)$, then $\Phi \in \mathtt{P}(\bbS^1)$ almost surely. 
\end{proposition}

\begin{proof}
	There exists a countable collection $\calI$ of intervals $I_0\subseteq \bbS^1$ such that every other interval contains at least one of them as a subset, so by (\ref{item:Kahane_ii}), $\Phi(\omega) \in \mathtt{P}(I(\omega)) \Rightarrow (\exists I_0 \in \calI) \Phi(\omega) \in \mathtt{P}(I_0)$ (where $I_0 = I_0(\omega)$ might depend on $\omega$). Consequently, 
	\begin{equation}
	0<\bbP[\Phi \in \mathtt{P}(I)] \leq \sum_{I_0 \in \calI} \bbP[\Phi \in \mathtt{P}(I_0)].
	\end{equation}
	This means that there exists some \emph{fixed} (non-random) $I_0 \in \calI$ such that $\Phi \in \mathtt{P}(I_0)$ with positive probability, so we can take $I=I_0$ (and in particular constant) without loss of generality. 
	By (\ref{item:Kahane_i}) and (\ref{item:Kahane_v}), $\mathtt{P}(I)$ is in the tail $\sigma$-algebra $\cap_{N>0} \sigma(a_n : |n|>N)$.
	This implies, by the Kolmogorov zero-one law, $\bbP [\Phi \in \mathtt{P}(I)]  \in \{0,1\}$, so given our choice of $I$, $\Phi \in  \mathtt{P}(I)$ almost surely. Pick $N\in \bbN$ such that $|I|> 2\pi / N$. Since by assumption $\Phi$ has symmetric and independent coefficients,
	\begin{equation}
	\Phi(\omega)  =  \sum_{n\in \bbZ} a_n(\omega) e^{in\theta} \quad \text{ and } \quad \Phi_M(\omega)  =  \sum_{n\in \bbZ} (-1)^{1_{n= M \bmod N}} a_n(\omega) e^{in\theta} 
	\end{equation} 
	are equidistributed, $M\in\{0,\ldots,N-1\}$. (Note that $\Phi_M$ is also measurable, and the term `equidistributed' has the obvious meaning.) We deduce from this that $\Phi_M \in \mathtt{P}(I)$ almost surely, and then from that (and the already made observation that $\Phi \in \mathtt{P}(I)$ almost surely) that $\Phi-\Phi_M \in \mathtt{P}(I)$ almost surely (using the subspace clause of (\ref{item:Kahane_i})). Applying (\ref{item:Kahane_iv}) a number of times, we conclude that 
	\begin{equation}
	\frac{1}{2}\cdot e^{-iM\theta} (\Phi-\Phi_M) = \sum_{n\in \bbZ, n= M \bmod N} a_n e^{i(n-M)\theta}
	\label{eq:3z}
	\end{equation}
	lies in $\mathtt{P}(I)$ almost surely. But \cref{eq:3z} is \emph{periodic} with period $2\pi / N$,  and so by (\ref{item:Kahane_iii})  lies in $\mathtt{P}(I+2\pi/N) \cap \cdots \cap\mathtt{P}(I+2\pi)$ almost surely. By the definition of $N$, the $N$ intervals $I+2\pi / N, \ldots, I+2\pi$ cover $\bbS^1$. Using (\ref{item:Kahane_ii} repeatedly, we conclude that \cref{eq:3z} lies in $\mathtt{P}(\bbS^1)$ almost surely. Adding together the $N$ different formal series \cref{eq:3z}, $M=0,\ldots,N-1$, (and using (\ref{item:Kahane_i}) one last time to say that the sum lies in $\mathtt{P}(\bbS^1)$,) we conclude that $\Phi \in \mathtt{P}(\bbS^1)$ almost surely. 
\end{proof}
\begin{remark}
	As an example of a property $\mathtt{P}$ that satisfies all of the hypotheses except (\ref{item:Kahane_i}), let $\mathtt{P}(I)$ denote the set of formal Fourier series which converge in $\scrD'(\bbS^1)$ to a distribution whose support is contained in the complement of $I$. Then, for $I$ a nonempty open interval, no nonzero trigonometric polynomials are in $\mathtt{P}(I)$. 
	
	Given any symmetric random $\bbR$-valued variable $a$, we can consider $\Phi = a$ as a random formal Fourier series (whose higher Fourier coefficients are all zero). Thus, $\Phi$ has independent and symmetric coefficients. If $a=0$ with positive probability, then $\Phi \in \mathtt{P}(\bbS^1)$ with positive probability, but if $\Phi \in \mathtt{P}(\bbS^1)$ almost surely then $a=0$ almost surely. 
	
	Hence, the proposition does not apply to this particular property $\mathtt{P}$.
\end{remark}
Our main argument is of a rather different sort, though we  make use of the trick (``rerandomization'') of randomly resampling the signs of symmetric random variables in several places.

\begin{remark}
	We obviously cannot dispense of the hypothesis that our coefficients are symmetric and independent, a restriction familiar from e.g.\ \cite{Nordlander}\cite{ItoNisio}\cite{Hoffmann1974}. 
\end{remark} 

We study ``Fourier series'' which are more specific than those in \cref{eq:rfs} in two ways: we restrict attention to Gaussian coefficients (although by appealing to domination results like \cite[Proposition 6.1.15]{Hytonen2016analysis} we could make do with Rademacher coefficients), and those coefficients obey upper and lower polynomial bounds. 
Given the generality of \Cref{prop:Kahane}, it seems plausible that our setup is overly restrictive, but we do not concern ourselves with this possibility here.

When considering regularity statements encodeable via almost sure membership in some Banach space, 
the geometry of that Banach space is relevant, specifically the theory of ($\gamma$-)radonifying maps into it. The $\gamma$-radonifying maps into $L^p$-spaces are closely related to $p$-summing maps.  This is remarked upon in \cite{SchwartzRadon}, but the consequences are only worked out systematically afterwards --- see, for example, either of \cite{SchwartzGeometry}\cite{SchwartzGeometrySummary}. A number of specialized treatments of summing and radonifying operators have appeared since the pioneering work in the 60's, 70's, and early 80's. Some emphasize probabilistic aspects, e.g.\ \cite{vanNeerven2010}, and some emphasize geometric and functional analytic aspects, e.g.\ \cite{Diestel95}.

Later papers on the regularity of  noise have typically focused on refinements of \cite{Reed74} involving $L^p$-based Sobolev spaces (or more generally Besov spaces), or they consider more general sorts of noises. As a sample:
\begin{enumerate}
	\item Kusuoka \cite{Kusuoka}, generalized Reed and Rosen's result to $L^p$-based Sobolev regularity (while keeping track of both logarithmic amounts of regularity and decay).
	\item Ciesielski, Kerkyacharian, and Roynette \cite{Ciesielski91}\Cite{Ciesielski93}\cite{Roynette}\cite{RoynetteSolo} and Hyt\"onen \& Veraar \cite{HytonenVeraar} considered various versions of the Besov regularity of (multidimensional) Brownian motion. 
	\item Fageot, Unser, and collaborators \cite{Fageot17}\cite{FageotPeriodic}\cite{Aziznejad2018wavelet} -- apparently with applications to image processing in mind -- considered the regularity of L\'evy noises, some of which were discussed earlier by Schwartz in \cite[Lecture 7]{SchwartzGeometry}.
	\item Veraar \cite{Veraar11}, in the process of understanding the periodic Korteweg-de Vries (KdV) equation, considered the  Besov regularity of white noise on the multidimensional torus. See also the discussion in \cite[Lecture 9]{SchwartzGeometry}.
\end{enumerate}

The point of restricting attention to the $L^2$-based Sobolev spaces, as we do here, is that we possess a complete classification of the radonifying operators between them, and more generally of the radonifying operators between any two separable Hilbert spaces. They are precisely the Hilbert-Schmidt (HS) operators. This is, in a nutshell, \emph{Sazonov's theorem}, along with its converse (which is actually also due to Sazonov \cite{Sazonov}). Hence, a canonical cylinder set measure (whose covariance is just the Hilbert space's inner product) on some separable Hilbert space is pushed forward by a bounded linear map to a Radon measure if and only if the map is Hilbert-Schmidt. See \cite[Part II-Chapter III, Theorem 2]{SchwartzRadon} for Sazonov's theorem proper and  \cite[Part II-Chapter IV, Theorem 3]{SchwartzRadon} for the converse. 
While we restrict attention to $L^2$-based Sobolev regularity,  we do so in geometrically nontrivial situations (as opposed to e.g.\ the flat torus), and with an eye towards locally and microlocally sharp \emph{ir}regularity statements. We also deal with fairly general Gaussian noises, though it should be clear that this added generality is largely cosmetic.

By ``microlocally sharp'' I intend to refer to the sharpening of 
\begin{itemize}
	\item a global irregularity statement, like 
	\begin{equation} 
	\text{``$\Phi(\omega) \notin \scrH^s(M)$ for $\bbP$-almost all $\omega\in \Omega$,''}
	\label{eq:unsharp_example} 
	\end{equation}
	where $(M,g)$ is our (compact) Riemannian manifold and $(\Omega,\calF,\bbP)$ is some probability space on which a random distribution $\Phi: \Omega\to \scrD'(M)$ is defined, to 
	\item a uniform and microlocal irregularity statement, like 
	\begin{equation} 
	``\operatorname{WF}^{s}(\Phi(\omega)) = \bbS^*M \text{ for $\bbP$-almost all $\omega\in \Omega$,''}
	\label{eq:sharp_example} 
	\end{equation} 
	where $\operatorname{WF}^s(\Phi(\omega))$ is the $s\in \bbR$ order Sobolev wavefront set of $\Phi(\omega)$, which we consider naturally as a subset of the cosphere bundle $\bbS^*M$ over $M$. 
\end{itemize}
Recall that  $\Phi(\omega) \in \scrH^s(M) \iff \operatorname{WF}^s(\Phi(\omega)) = \varnothing$, which is why \cref{eq:sharp_example} is a sharpening of \cref{eq:unsharp_example}.
For the probabilist to whom microlocal analysis might be unfamiliar, we provide a brief overview of the Kohn-Nirenberg calculus in \S\ref{subsec:KN}. Alternatively, accessible expositions can be found in \cite{Grigis}\cite{Vasy18}\cite{HintzNotes}, and the standard reference is H\"ormander's \cite[Chapter 18]{Hormander}. We refer the reader to these resources for any undefined notation related to the pseudodifferential calculi herein.

Now, it is straightforward to construct Gaussian noise on compact Riemannian manifolds, and the \emph{global} $L^2$-based Sobolev regularity is readily and elementarily computed. See \S\ref{subsec:random}. Very little is needed. What is not so clear, and seems to require more technology, is that the local (and microlocal) regularity at any specified point is no better than the global regularity. 
Via our approach -- and restricting attention to local regularity for the moment -- this statement is reduced via Sazonov's theorem (see \Cref{Sazonov}) to estimating the Hilbert-Schmidt norm of the Kohn-Nirenberg pseudodifferential operators 
\begin{equation}
(1+\triangle_g)^{\sigma/2} M_\chi (1+\triangle_g)^{\varsigma/2} 
\label{eq:lh03}
\end{equation}
as acting on the Hilbert space $\scrL^2(M,g)$, if they act on $\scrL^2(M,g)$ at all. Here $\sigma,\varsigma\in \bbR$, and $M_\chi$ is pointwise multiplication by an arbitrary ``bump'' function $\chi \in C^\infty(M)$. Later on, we will consider \cref{eq:lh03} for more general \emph{microlocal cutoffs} in place of $M_\chi$, and this is what allows us to get microlocal statements. The fractional powers of the Laplacian appearing in \cref{eq:lh03} are discussed briefly in \S\ref{subsec:spectral}, and the reader is directed to \cite{Shubin} for details.

The structure of this paper is as follows. 
\begin{itemize}
	\item In \S\ref{sec:setup}, as indicated above, we set up the problem and review the Kohn-Nirenberg pseudodifferential operators ($\Psi$DOs). In \S\ref{subsec:random}, we discuss the global $L^2$-Sobolev regularity of random distributions. All in all, \S\ref{sec:setup} is essentially entirely expository. 
	\item In \S\ref{sec:sazonov}, we prove a consequence, \Cref{Sazonov}, of Sazonov's theorem which reduces questions of pathwise $L^2$-based regularity to functional analytic estimates. In the process, we discuss the various senses in which a random Hilbert space-valued series (with Gaussian coefficients) can converge and exposit their (well-known  \cite[Lecture 1]{SchwartzGeometry}) equivalence, \Cref{prop:Sazonov_equiv}. 
	
	Partly for the sake of completeness, I've attempted to make this section rather thorough in terms of the different ways in which the result can be deduced from other famous probabilistic theorems.  We also state an It\^o-Nisio-type  theorem, \Cref{Ito-Nisio}, essentially a special case of a result of Hoffmann-J{\o}rgensen \cite{Hoffmann1974}, which ought to be useful in the study of random distributions more generally.  Nothing in this section is essentially new, but hopefully some readers will find some utility in our overview, and our particular presentation lends itself to the material in \S\ref{sec:main_theorem}. 
	\item In \S\ref{sec:main_theorem}, we apply as a blackbox the consequence of Sazonov's theorem to the pathwise microlocal regularity of Gaussian noise. Our main theorem is \Cref{thm:main}, and we deduce from it \Cref{thm:prelim}. 
\end{itemize}
It's worth noting that the core of our argument, \S\ref{sec:main_theorem}, is quite easy, once the general results of \S\ref{sec:setup} and \S\ref{sec:sazonov} are at our disposal, and this is one attractive feature of the approach. It seems likely that the main theorem could be deduced via more direct arguments, but they would likely not possess the simplicity of the argument here.

While our scope in this paper is rather narrow, the same sort of analysis is expected to apply (\emph{mutatis mutandis}) more generally:
\begin{enumerate}
	\item While we have singled out the Laplace-Beltrami operator $\triangle_g$, analogous arguments apply to other semibounded self-adjoint elliptic pseudodifferential operators.
	\item While we only consider random distributions, i.e.\ rough sections of the trivial bundle over $M$, the arguments can also handle random distributional sections of arbitrary vector bundles (with $\triangle_g$ replaced appropriately).
	\item While we only deal with $L^2$-based regularity, using the theory of $p$-summing and radonifying operators in conjunction with the $L^p$-boundedness of zeroth order Kohn-Nirenberg pseudodifferential operators it should be possible to get sharpenings of \Cref{thm:main} in the $L^p$-based Sobolev spaces.
	\item Finally, while we have restricted attention to compact manifolds, it is not significantly more effort to handle asymptotically conic manifolds, in which case we need to work with scattering $\Psi$DOs in the sense e.g.\ of Melrose \cite{Melrose}. The results of Reed \& Rosen  and Kusuoka in the case of exact Euclidean space fit into this framework. 
\end{enumerate}
It's also worth noting that it is likely possible to prove that, in a precise sense, given most natural notions of regularity, Gaussian noise on the torus and on a general compact Riemannian manifold have the same level of regularity. ``Natural'' here roughly means invariant under diffeomorphisms and 0th order Kohn-Nirenberg $\Psi$DOs. We do not address here the possibility of deducing (ir)regularity results on $(M,g)$ from the corresponding results on the torus $\bbT^d$, but this possibility seems promising, albeit not obviously more efficient than the direct argument in \S\ref{sec:main_theorem}, at least in the $L^2$-based case. 
These extensions will be left to possible future works.

\section{The setup}
\label{sec:setup}
Suppose that $(M,g)$ is a (nonempty, connected) compact Riemannian manifold -- without boundary and of course smooth -- of positive dimension $d=\dim M$, and let $\triangle_g \in \operatorname{Diff}^2(M)$ denote the positive-semidefinite Laplace-Beltrami operator corresponding to the metric $g$. In local coordinates $x^\bullet: M\supseteq U \to \bbR^d$, $\triangle_g$ is given by 
\begin{equation}
\triangle_g = - \frac{1}{|g|^{1/2}} \sum_{i,j=1}^d \frac{\partial}{\partial x^i} (|g|^{1/2} g^{ij}  \frac{\partial}{\partial x^j}),
\end{equation} 
where $g^{ij}$ is the inverse metric and $|g|= \det (g_{ij})_{i,j}$. The \emph{cosphere bundle} $\bbS^* M$ over $M$ is the boundary of the radial compactification of the cotangent bundle over $M$. 

\subsection{Function Spaces and the Kohn-Nirenberg Calculus}
\label{subsec:KN}

We will denote by 
\begin{equation} 
\scrD(M) = \{C^\infty \text{ functions }\varphi: M\to \bbC \}
\label{eq:test_functions}
\end{equation}  
the complex Fr\'echet space of smooth $\bbC$-valued functions on $M$, whose topology is generated by a $g$-dependent choice of $C^k$ norms for $k\in \bbN$. The norms are $g$-dependent, but the topological vector space (TVS) $\scrD(M)$ is not. The function spaces we consider do not depend on the metric $g$ at the level of sets, and so we will often omit the metric dependence from our notation, at least when the metric dependence is unimportant.
I will denote by $\scrD'(M)$ the TVS-dual of $\scrD(M)$, endowed with the weak-$*$ topology.
This is (abstractly) the LCTVS of distributions on $M$. 

The duality pairing here will be formally written like the $\scrL^2(M,g)$-inner product, 
\begin{align}
\langle - ,- \rangle_{\scrL^2(M,g)}&:\scrD(M)\times \scrD'(M) \ni (\varphi,u) \mapsto \int_M \varphi(x)^* u(x) \dd^d \operatorname{Vol}_g(x) = u(\varphi^*),
\label{eq:duality_pairing} \\
\langle - ,- \rangle_{\scrL^2(M,g)}&:\scrD'(M)\times \scrD(M) \ni (u,\varphi) \mapsto \int_M u(x)^* \varphi(x) \dd^d \operatorname{Vol}_g(x) = u(\varphi^*)^*.
\label{eq:duality_pairing_ii}
\end{align}
No confusion should arise out of the overloaded notation.
The metric $g$ defines a trivialization of $\Omega^1(M)$, the 1-density bundle over $M$, and this induces a $g$-dependent antilinear embedding  
\begin{equation}
\iota_g: \scrD(M) \hookrightarrow \scrD'(M)
\label{eq:function_embedding} 
\end{equation}
of topological vector spaces, $\iota_g(\varphi) = \langle  \varphi,- \rangle_{\scrL^2(M,g)}$. 
We typically identify $\scrD(M)$ with its image under $\iota_g$. Subspaces of $\scrD'(M)$ will be denoted with the mathscr typeface. 

Recall that, for each $s\in \bbR$, we can define a Sobolev space $\scrH^s(M)$ of distributions on $M$ which are locally in $\scrH^s(\bbR^d)$. 
This notion is well-defined, and
naturally a $g$-independent Hilbertizable space. These are indexed in order of increasing regularity, so 
\begin{equation}
\scrD'(M)``=" \scrH^{-\infty}(M) \supsetneq \cdots  \supsetneq \scrH^s(M) \supsetneq \scrH^{s'}(M) \supsetneq \cdots \supsetneq \scrH^{+\infty} ``=" \scrD(M) 
\label{eq:0b3}
\end{equation}
for any $s< s'$. Here we are identifying smooth functions on $M$ with their equivalence classes of functions agreeing almost everywhere with respect to the Lebesgue measure on coordinate charts.
The inclusion \cref{eq:function_embedding} induces a continuous embedding 
\begin{equation} 
\iota_g: \scrH^s(M)\hookrightarrow \scrD'(M), 
\end{equation} 
and we identify each Sobolev space with its image under $\iota_g$ in \cref{eq:0b3} and below. 
Given a metric $g$, we can define a natural inner product $\langle - ,- \rangle_{\scrH^s(M,g)}$ on $\scrH^s(M)$ by writing 
\begin{equation}
\langle u, v \rangle_{\scrH^s(M,g)} = \langle (1+\triangle_g)^{s/2} u , (1+\triangle_g)^{s/2} v \rangle_{\scrL^2(M,g)}
\end{equation}
for $u,v \in \scrH^s(M)$. 
The resultant $g$-dependent Hilbert space will be denoted $\scrH^s(M,g)$.
(See below for a discussion of $(1+\triangle_g)^{s/2}$.)

Central to this paper is the graded $\bbC$-algebra 
 \begin{equation}
 \Psi (M) = \cup_{s\in \bbR} \Psi^s(M) 
 \end{equation}
 \begin{equation}
 \Psi^{-\infty}(M) \subsetneq \cdots \subsetneq \Psi^s(M) \subsetneq \cdots \subsetneq \Psi^{s'}(M) \subsetneq \cdots \subsetneq \Psi^{+\infty}(M) 
 \label{eq:ps}
 \end{equation}
 ($s,s' \in \bbR$, $s<s'$) of \emph{Kohn-Nirenberg} pseudodifferential operators on our compact manifold $M$, as defined e.g.\ in \cite[Chapter XVIII]{Hormander}\cite{Grigis}\cite{Vasy18}\cite{HintzNotes}. 
 The elements of $\Psi(M)$ will be thought of as particular continuous linear maps $\scrD(M)\to \scrD(M)$ that extend to continuous linear maps $\scrD'(M)\to \scrD'(M)$. 
 The ur-example of an element of the vector space $\Psi^s(M)$ is the essentially-self adjoint operator 
 \begin{equation}
 (1+\triangle_g)^{s/2} : \scrD(M)\to \scrD(M)\subseteq \scrL^2(M,g) 
 \label{eq:fractional_power_laplacian}
 \end{equation}
 on $\scrL^2(M,g)$ defined via the functional calculus. The fact that the fractional powers of the Laplacian are elements of $\Psi(M)$ of the indicated orders is a special case of a well-known theorem of Seeley \cite{SeeleyComplex}\cite{SeeleyComplex2}.
 This has several nontrivial consequences  -- e.g.\ pseudolocality (and microlocality) -- which are not obvious from the spectral theoretic construction of \cref{eq:fractional_power_laplacian}. One possible explanation for the brevity of the analysis in this paper is the power of Seeley's theorem.
 
 Kohn-Nirenberg $\Psi$DOs on manifolds are given locally, meaning modulo globally defined smoothing operators, by Kohn-Nirenberg $\Psi$DOs on $\bbR^d$. For $u\in \scrS(\bbR^d)$ a Schwartz function, the action of a Kohn-Nirenberg operator $A \in \Psi^s(\bbR^d)$, $s\in \bbR$, on $u$ is given by the well-defined iterated integral
\begin{equation}
 A u(x) =  \frac{1}{(2\pi)^d} \int_{\bbR^d} e^{ix\cdot \xi} a(x,\xi)\Big( \int_{\bbR^d} e^{-iy\cdot \xi}  u(y) \dd^d y \Big)\dd^d \xi = \frac{1}{(2\pi)^d} \int_{\bbR^d} e^{ix\cdot \xi} a(x,\xi) \hat{u}(\xi) \dd^d \xi
 \label{eq:ll}
\end{equation}
where $a\in C^\infty(\bbR^{2d})$ is the \emph{full symbol} of $A$. In other words, $\Psi^s(\bbR^d)$ can be defined as the set of operators on $\scrS(\bbR^d)$ given by \cref{eq:ll} for $a \in S^s(\bbR^d,\bbR^d)$, the set of Kohn-Nirenberg symbols. This is the ``quantization'' typically known as ``left quantization.'' Other possible quantizations (all equivalent to left quantization by the so-called reduction formula) are Weyl quantization (which seems to be the most common in the literature) and right quantization. The Kohn-Nirenberg symbols are smooth functions satisfying polynomial asymptotics at fiber infinity (in a sense made precise by radial compactification), in the sense that 
\begin{equation}
\sup_{x\in \bbR^d, \xi \in \bbR^d} | \langle \xi \rangle^{-s+|\beta|} \partial_x^\alpha \partial_\xi^\beta  a(x,\xi) | < \infty 
\label{eq:j9}
\end{equation}
for any multi-indices $\alpha,\beta \in \bbN^d$, $|\beta|= \lVert \beta \rVert_{\ell^1}$. Here $x$ is the ``base coordinate'' and $\xi$ is the ``fiber coordinate.'' After a bit of work (meaning a few estimates), the details of which can be found in the references cited above, the action of $A$ on $\scrS'(\bbR^d)$ can be defined via continuity. 

Here are a few key properties of the algebra of Kohn-Nirenberg $\Psi$DOs, the proofs of which can be found in the references cited above or a myriad of other places, in some form or another. I will use the following results mostly without comment (or with at most a brief comment) in \S\ref{sec:main_theorem}. First, $\Psi(M)$ is a graded $\bbC$-algebra, with 
 	\begin{equation}
 	\Psi^s(M) \Psi^\ell(M) = \{P_1P_2: P_1\in \Psi^s(M), P_2\in \Psi^\ell(M)\} = \Psi^{s+\ell}(M) 
 	\end{equation}
 	for all $s,\ell \in \bbR$, where $\Psi^s(M)$ denotes the $s$th level of $\Psi(M)$. (In fact, $\Psi(M)$ is a graded Fr\'echet algebra, since the symbol spaces are Fr\'echet spaces whose algebraic structure is compatible with the algebraic structure of the calculus of $\Psi$DOs.) The $C^\infty(M)$-module $\operatorname{Diff}^k(M)$ of $k$th order ($k\in \bbN$) differential operators on $M$ satisfies 
 	\begin{equation}
 	\operatorname{Diff}^k(M) \subseteq \Psi^k(M). 
 	\end{equation}
 	Moreover, if $A \in \operatorname{Diff}^k(M)$ is invertible and elliptic (see below), then we may consider its inverse as an element $A^{-1} \in \Psi^{-k}(M)$. (This observation can be extended to generalized inverses in a straightforward way -- cf.\ \cite[Theorem 5.45]{HintzNotes}.) 
 	In particular, $(1+\triangle_g)^k \in \Psi^{2k}(M)$ for all $k\in \bbZ$. (And then Seeley's theorem extends this to the $k\notin \bbZ$ case.)

 	If $A : \scrD(M)\to \scrD'(M)$ is a continuous linear map with a smooth Schwartz kernel, then $A \in \Psi^s(M)$ for all $s\in \bbR$, and in fact 
 	\begin{equation} 
 	A:\scrD'(M)\to \scrD(M). 
 	\label{eq:sm}
 	\end{equation} 
 	Writing $\Psi^{-\infty}(M) = \cap_{s\in \bbR} \Psi^s(M)$ as in \cref{eq:ps}, $A \in \Psi^{-\infty}(M)$. Such $A$ are often called \emph{residual} or, by virtue of \cref{eq:sm}, \emph{smoothing}. For each $s\in \bbR$, we have a short exact sequence 
 	\begin{equation}
 	0 \to \Psi^{s-1}(M)\subseteq \Psi^s(M) \overset{\sigma^s}{\to} S^{[s]}(T^* M) \to 0, 
 	\label{eq:ses}
 	\end{equation}
 	where $S^{[s]}(T^* M) = S^s(T^* M)/S^{s-1}(T^* M)$ and $S^s(T^* M)$ is the Fr\'echet space of Kohn-Nirenberg symbols $T^* M\to \bbC$. This is the set of smooth functions satisfying all the estimates \cref{eq:j9} in local coordinates. Also, $\sigma^s$ is the \emph{principal symbol map}. Locally, $\sigma^s$ is given by $A\mapsto a \bmod S^{s-1}(\bbR^d,\bbR^d)$, where $A,a$ are as in \cref{eq:ll}. 
 	Moreover: given $P \in \Psi^s(M), Q \in \Psi^\ell(M)$, 
 	\begin{equation}
 	\sigma^s(P) \sigma^\ell(Q) = \sigma^{s+\ell}(PQ).  
 	\label{eq:pq}
 	\end{equation} 
 	That is, the principal symbol map is a  multiplicative homomorphism. (This is, in some sense, \emph{the} key property of the symbol calculus.) Properly speaking, principal symbols are not functions but equivalence classes of functions. When we write a statement like \cref{eq:pq}, we implicitly mean that the indicated operation -- which makes sense for Kohn-Nirenberg symbols -- induces a well-defined operation on elements of $S^{[s]}(T^* M)$. This abuse of terminology is standard and mostly harmless, as far as I can tell. In \S\ref{sec:main_theorem}, we will work instead with zeroth order classical symbols, which can be considered as functions, but on the cosphere bundle over $M$ rather than the cotangent bundle. Since classical symbols can also be thought of as ordinary symbols, we will not belabor the point with a discussion of their particular properties. 
 	
 	Suppose that $P \in \Psi^s(M)$ is \emph{elliptic} at a point $(x_0,\xi_0) \in \bbS^* M$, $x_0\in M,\xi_0 \in \bbS^*_x M$. This means that we have a lower bound 
 	\begin{equation} 
 		\varsigma^s(P)(x,\xi) \geq c\langle \xi \rangle^s
 	\end{equation}  
 	for some $c>0$ and  all sufficiently large $\xi$, for $(x,\xi)$ in some conic neighborhood $\Gamma \subset T^* X$ of $(x_0,\xi_0)$. Here $\varsigma^s(P)$ is an arbitrary representative of the equivalence class $\sigma^s(P)$.  (Note that this notion is well-defined, so both independent of a choice of coordinates and of the choice of representative.)
 	
 	Then there exists a ``(microlocal) parametrix'' $Q\in \Psi^{-s}(M)$ which microlocally inverts $P$ near $(x,\xi)$ in the sense that
 	\begin{equation}
 	(x,\xi)\notin \operatorname{WF}'(PQ-I) , \operatorname{WF}'(QP-I), 
 	\end{equation} 
 	where $I\in \operatorname{Diff}^0(M)$ is the identity. 
 	Here $\operatorname{WF}'$ denotes the essential support of a $\Psi$DO, roughly the set of points in the cosphere bundle where our operator is not smoothing. This is most elementarily defined in local coordinates: given $A \in \Psi^s(\bbR^d)$ with full symbol $\sigma(A) \in \smash{S^s(\bbR^d_x,\bbR^d_\xi)}$, $(x,\xi) \notin \operatorname{WF}'(A)$ if and only if $a$ vanishes rapidly as $\xi\to \infty$ (or equivalently is of infinitely negative symbolic order) in some conic neighborhood of $(x,\xi)$, and this notion is diffeomorphism invariant and hence extends to a well-defined notion regarding compact manifolds. 
 	Alternatively, one can define $\operatorname{WF}'(A)$ in terms of the wavefront set of the Schwartz kernel $K_A$, or by demoting the statement of microlocality to its definition.  Cf.\ \cite[Chapter 18, Proposition 18.1.26]{Hormander}. Given any $A \in \Psi^s(M)$, 
 	\begin{equation} 
 	\operatorname{WF}^{s_0-s}(A \varphi) \subseteq \operatorname{WF}^{s_0}(\varphi) \cap \operatorname{WF}'(A) 
 	\label{eq:ul}
 	\end{equation} 
 	for all $\varphi \in \scrD'(M)$ and $s_0 \in \bbR$. (See below for the definition of $\operatorname{WF}^s$.) Moreover/in particular, $A:\scrH^{s_0}(M)\to \scrH^{s_0-s}(M)$ is continuous for all $s_0 \in \bbR$. Given any $A\in \Psi^s(M)$, 
 	\begin{equation} 
 	\operatorname{WF}^{s_0-s}(A\varphi) \supseteq \operatorname{WF}^{s_0}(\varphi) \backslash \operatorname{Char}^s(A), 
 	\label{eq:er}
 	\end{equation} 
 	where $\operatorname{Char}^s(A) \subseteq \bbS^* M$ is the \emph{characteristic set} of $A$, i.e.\ the set of points in the cosphere bundle at which $A$ fails to be elliptic.

These properties are interrelated: elliptic regularity, for example, can be derived from microlocality and the construction of an elliptic parametrix. We have omitted properties which, while central to applications of the Kohn-Nirenberg calculus to PDE, do not bear on the discussion in \S\ref{sec:main_theorem}.
Recall that the $\scrH^s$-wavefront set of $u\in \scrD'(M)$ is defined (or can be defined) as 
\begin{equation}
\operatorname{WF}^s(u) = 
\bigcap \{\operatorname{Char}^0(A) : A \in \Psi^0(M), Au \in \scrH^s(M)\}.
\label{eq:wf}
\end{equation}
The properties \cref{eq:ul}, \cref{eq:er} furnish the interpretation of $\operatorname{WF}^s(u)$ as the points in $\bbS^* M$ representing locations at which and directions in which $u$ possesses singularities obstructing the potential inclusion $u \in \scrH^s(M)$. Given $u \in \scrS'(\bbR^d)$, $\operatorname{WF}^s(u)$ is defined analogously to \cref{eq:wf}, and it is straightforward to demonstrate that 
\begin{equation}
\bbS^* \bbR^d\backslash\! \operatorname{WF}^s(u)  = \{(x,\hat{\xi}) \in \bbR^{d}_x\times \bbS^{d-1}_{\hat{\xi}} : \exists \chi_1 \in \scrC_{\mathrm{c}}^\infty(\bbR^d), \chi_2 \in \scrC_{\mathrm{c}}^\infty(\bbS^{d-1}) \text{ s.t. \cref{eq:cc}}\},
\end{equation}
\begin{equation}
\chi_1(x) = 1, \quad  \chi_2(\hat{\xi}) = 1, \quad \text{ and } \quad \int_{\bbR^d} e^{ix\cdot \xi } \tilde{\chi}_2(\xi)\widehat{\chi_1 u}(\xi) \dd^d \xi \in \scrH^s(\bbR^d_x), 
\label{eq:cc}
\end{equation}
where $\tilde{\chi}_2 \in \scrC^\infty(\bbR^d_\xi)$ is given by $\tilde{\chi}_2(\xi) = \chi_2(\hat{\xi}) \chi_3(\lVert \xi \rVert^2)$, $\xi=\hat{\xi} \lVert \xi \rVert$, $\chi_3 \in C^\infty(\bbR)$ vanishing in some neighborhood of zero and identically equal to one outside of some slightly larger neighborhood. \cref{eq:wf} agrees locally with the definition in Euclidean space, in the obvious sense. Note that via Parseval-Plancherel, the last condition in \cref{eq:cc} is equivalent to $\smash{\langle \xi \rangle^{+s} \tilde{\chi}_2(\xi) \widehat{\chi_1 u}(\xi) \in \scrL^2(\bbR^d_\xi)}$, so $\operatorname{WF}^s(u)$ captures the locations $x$ at which and directions $\smash{\hat{\xi}}$ in which the Fourier transform of $u$ (suitably cutoff) fails, 
in any conic neighborhoods of $(x,\smash{\hat{\xi}})$,
\begin{itemize}
	\item ($s\geq 0$) to decay at least as fast as $\langle \xi \rangle^{-s}$ 
	\item ($s\leq 0$) to grow no faster than $\langle \xi \rangle^{-s}$
\end{itemize}
in an $L^2$-averaged sense. $\operatorname{WF}^s(u)$ is a refinement of the $s$th order singular support, 
\begin{equation} 
\operatorname{singsupp}^s(u) =M\backslash\{x\in M : \exists \chi \in \scrC^\infty(M), \chi(x) = 1, \chi u \in \scrH^s(M)\},
\end{equation} 
of $u$ 
in the sense that $\pi(\operatorname{WF}^s(u)) = \operatorname{singsupp}^s(u)$, where $\pi:\bbS^* M\to M$ is the canonical projection. 

\subsection{The Spectral Decomposition of $\triangle_g$}
\label{subsec:spectral}

Standard elliptic theory -- of which this is the ur-example, in some form going all the way back to Weyl -- tells us that for each $f\in \scrD(M)$, any distributional solution $u \in \scrD'(M)$ to the PDE 
\begin{equation}
\triangle_g u  = f 
\end{equation}
is necessarily smooth, i.e.\  $u \in \scrD(M)$. (This follows from elliptic regularity and the Sobolev embedding theorems.) In concert with analytic Fredholm theory -- of which the spectral family of the Laplacian is also the ur-example -- this tells us that there exist smooth functions $\phi_0=(\operatorname{Vol}_g(M))^{-1/2},\phi_1,\phi_2,\ldots \in \scrD(M)$ such that 
\begin{equation}
\{\phi_n\}_{n=0}^\infty \subset \scrL^2(M,g) 
\label{eq:basis} 
\end{equation}
is an $\scrL^2(M,g)$-orthonormal basis of eigenvectors of $\triangle_g$, where the eigenvalues possibly occur with (finite) multiplicity and no accumulation points. The Laplace-Beltrami operator is positive semidefinite and self-adjoint, so its spectrum $\sigma(\triangle_g)$ satisfies $\sigma(\triangle_g)\subseteq [0,\infty)$.  Reindexing \cref{eq:basis} if necessary, we can therefore write 
\begin{equation}
\sigma(\triangle_g) = \{0 = \lambda_0 \leq \lambda_1 \leq \lambda_2 \leq \lambda_3 \leq \cdots \},
\end{equation}
where $\triangle_g \phi_n = \lambda_n \phi_n$ and $\lim_{n\to\infty} \lambda_n=\infty$. We use the physicists' convention in starting our indexing with zero. ($\phi_0$, which is just a constant, is the ground state.) 

Stronger information on the asymptotic distribution of the $\lambda_n$ is given by \emph{Weyl's law}, which (in its most unrefined form) states that the function $N:[0,\infty)\to \bbN$ given by $N(\lambda)=\#\{n\in \bbN : \lambda_n \leq \lambda\}$ has the asymptotics 
\begin{equation}
N(\lambda) = (1+o(1)) \frac{ \operatorname{Vol}(\bbB^d)}{(2\pi)^d}  \operatorname{Vol}_g(M)\lambda^{d/2}
\label{eq:Weyl's law}
\end{equation}
as $\lambda \to \infty$, where $\operatorname{Vol}(\bbB^d)$ is the volume of the unit ball in $\bbR^d$. See \cite[\S14.3.4]{Zworski}. This result is discussed in most introductory accounts of the applications of pseudodifferential operators to spectral geometry. Note that we can invert \cref{eq:Weyl's law} to give an asymptotic formula for $\lambda_n$:

\begin{proposition}
	\label{Weyl's_law}
	Given any compact Riemannian manifold $(M,g)$, the $\lambda_n$ defined above satisfy $
	\lambda_n = (1+o(1)) 4\pi^2(\operatorname{Vol}(\bbB^d)\operatorname{Vol}_g(M))^{-2/d}  n^{2/d}$
	as $n\to \infty$. 
\end{proposition}
\begin{proof} 
	Since $\lambda_n\to \infty$ as $n\to\infty$, \cref{eq:Weyl's law} holds with $\lambda_n$ substituted in for $\lambda$, where the asymptotics are now to be understood as being taken as  $n\to\infty$: 
	\begin{equation}
	N(\lambda_n) = (1+o(1))  \frac{ \operatorname{Vol}(\bbB^d)}{(2\pi)^d} \operatorname{Vol}_g(M) \lambda_n^{d/2}.
	\label{eq:Weyl's law inverted preliminarily}
	\end{equation}
	Weyl's law implies that the multiplicity of any given eigenvalue $\lambda_n$ is  $\lim_{\epsilon\to 0^+} (N(\lambda_n) - N(\lambda_n-\epsilon)) = \smash{o(\lambda_n^{d/2})}$ as $n\to\infty$, so that 
	\begin{equation} 
	N(\lambda_n)=n+o(\lambda_n^{d/2}).
	\label{eq:misc_019}
	\end{equation} 
	(Indeed, by Weyl's law, $N(\lambda_n-\varepsilon) = (1+o(1))(2\pi)^{-d} \operatorname{Vol}(\bbB^d) \operatorname{Vol}_g(M)(\lambda_n-\varepsilon)^{d/2}$. Subtracting this from \cref{eq:Weyl's law inverted preliminarily} and taking $\epsilon\to 0^+$, \cref{eq:misc_019} follows.)
	 Plugging this into  (\ref{eq:Weyl's law inverted preliminarily}) and solving for $\lambda_n$ yields the result. 
\end{proof}

For each $\varsigma \in \bbR$, Let $h^\varsigma(\bbN)$ denote the Hilbert space of all sequences  $\bfa=\{a_n\}_{n=0}^\infty \subset \bbC$ such that 
\begin{equation}
\{ (1+n)^\varsigma a_n\}_{n=0}^\infty \in \ell^2(\bbN), \text{ i.e. s.t. } \lVert \{a_n\}_{n=0}^\infty \rVert_{h^\varsigma}^2 \overset{\mathrm{def}}{=}\sum_{n=0}^\infty (1+n)^{2\varsigma} |a_n|^2 < \infty,
\end{equation}
with the usual inner product $\langle \bfa,\bfb \rangle_{h^\varsigma} = \langle (1+n)^{\varsigma} \bfa, (1+n)^{\varsigma} \bfb \rangle_{\ell^2}$.
Similarly, define 
\begin{equation}
d(\bbN) = \cap_{\varsigma \in \bbR}h^\varsigma(\bbN) \quad \text{ and } \quad  d'(\bbN) = \cup_{\varsigma \in \bbR} h^{\varsigma}(\bbN). 
\end{equation} 
Endow $d(\bbN)$ with the topology generated by the countably many norms $\lVert - \rVert_{h^k}$ for $k\in \bbZ$. $d(\bbN)$ is the Fr\'echet space consisting of all superpolynomially decaying sequences of complex numbers. Given $\bfa=\{a_n\}_{n=0}^\infty \in d'(\bbN)$ and $\bfb=\{b_n\}_{b=0}^\infty \in d(\bbN)$, define 
\begin{equation} \langle \bfa,\bfb \rangle_{\ell^2} = 
\langle \{a_n\}_{n=0}^\infty , \{b_n\}_{n=0}^\infty \rangle_{\ell^2} = \sum_{n=0}^\infty a_n^* b_n \in \bbC.
\label{eq:sequential_duality_pairing} 
\end{equation}
I will call this the $\ell^2$-pairing. The set of functions $d(\bbN)\to \bbC$ of the form $\langle \{a_n\}_{n=0}^\infty ,-\rangle_{\ell^2}$ for some $\{a_n\}_{n=0}^\infty \in d'(\bbN)$ is precisely the set of continuous linear functionals on $d(\bbN)$, so \cref{eq:sequential_duality_pairing} yields an identification $d'(\bbN) \cong d(\bbN)^* $
of  $d'(\bbN)$ and the underlying vector space of the LCTVS-dual of $d(\bbN)$. We correspondingly endow the vector space $d'(\bbN)$ with the compatible weak-$*$ topology.  

\begin{proposition}
	\label{sequence_Sobolev_correspondence}
	For each $\varsigma\in \bbR$ and $s \leq \varsigma d$, we have a ($g$-dependent) continuous linear map $\Sigma: h^\varsigma(\bbN)\to \scrH^s(M)$ given by
	\begin{equation}
	\{a_n\}_{n=0}^\infty \mapsto \sum_{n=0}^\infty a_n \phi_n , 
	\label{eq:generic_sum}
	\end{equation}
	the series on the right-hand side being summable in $\scrH^s(M)$. This is an isomorphism of vector spaces if (and only if) $\varsigma = s/d $, in which case it is an equivalence of Banach spaces. 
\end{proposition}
\begin{proof}
	Since $\scrH^s(M)$ is complete, the partial series corresponding to the series on the right-hand side of \cref{eq:generic_sum} converge to some element in $\scrH^s(M)$ in the given topology if and only if 
	\begin{equation}
	\sum_{n=0}^\infty |a_n|^2 (1+\lambda_n)^s  < \infty. 
	\label{eq:convergence_estimate} 
	\end{equation}
	Via Weyl's law in the form of \Cref{Weyl's_law}, we see that there exist some $(M,g)$-dependent and $s$-dependent constants $c,C,N_0>0$ such that
	\begin{equation} 
	cn^{2s/d}\leq (1+\lambda_n)^s \leq Cn^{2s/d} 
	\label{eq:cC}
	\end{equation} 
	for all $n>N_0$. So \cref{eq:convergence_estimate} holds if (and only if) $\{a_n\}_{n=0}^\infty \in h^{s/d}(\bbN)$. 	
	So $\Sigma$ does indeed map $h^\varsigma(\bbN)\to \scrH^s(M)$ when $s\leq \varsigma d$.
	The upper estimate also shows that $\lVert \Sigma \{a_n\}_{n=0}^\infty \rVert_{\scrH^s} \leq C' \lVert \{a_n\}_{n=0}^\infty \rVert_{h^\varsigma}$ for those $s,\varsigma$ and some $C'>0$, so that $\Sigma$ is continuous. 
	
	$\Sigma$ is clearly injective. The surjectivity of $\Sigma$ for $\varsigma = s/d$ follows from the completeness of $\phi_0,\phi_1,\ldots$ as an orthonormal basis of $\scrL^2(M,g)$: given $u \in \scrH^s(M)$, there exist $b_0,b_1,b_2,\ldots \in \bbC$ such that $\{b_n\}_{n=0}^\infty \in \ell^2(\bbN)$ and 
	\begin{equation}
	\sum_{n=0}^\infty b_n \phi_n = (1+\triangle_g)^{s/2} u. 
	\end{equation}
	By the continuity of $(1+\triangle_g)^{-s/2} : \scrL^2(M,g)\to \scrH^s(M,g)$, this implies that $
	\sum_{n=0}^\infty (1+\lambda_n)^{-s/2} b_n \phi_n = u$. Set $a_n = (1+\lambda_n)^{-s/2}b_n$. By \cref{eq:cC}, $\bfa = \{a_n\}_{n=0}^\infty$ lies in $h^{s/d}(\bbN)$ and satisfies $\Sigma \bfa = u$. We conclude that $\Sigma$ is surjective, and (e.g.\ by the open mapping theorem or direct estimate)  an equivalence of Banach spaces. 
\end{proof}

\begin{corollary}
	$\Sigma$ yields isomorphisms 
	\begin{equation}
	\Sigma|_{d(\bbN)} : d(\bbN)\to \scrD(M)  \quad \text{ and } \quad \Sigma : d'(\bbN)\to \scrD'(M) 
	\end{equation}
	of TVSs compatible with the duality pairings \cref{eq:sequential_duality_pairing} and \cref{eq:duality_pairing}. 
\end{corollary}

\subsection{Random Distributions} 
\label{subsec:random}
A \emph{random distribution}, also known as a $\scrD'(M)$-valued random variable, consists of the following data: a (not necessarily complete) probability space $(\Omega,\calF,\bbP)$ and a $(\calF,\operatorname{Borel}(\scrD'(M)))$-measurable function $\Phi :\Omega\to \scrD'(M)$. 
We will use the symbol `$\Phi$', with or without subscripts, to denote random distributions. The pushforward 
\begin{equation} 
\Phi_* \bbP:\operatorname{Borel}(\scrD'(M)) \to [0,1], \quad 
\Phi_* \bbP(S) = \bbP (\Phi^{-1} (S)),
\end{equation} 
of $\bbP$ by a random distribution $\Phi$ is a Borel probability measure on $\scrD'(M)$ called the \emph{law} of $\Phi$. Probabilists often use the word `distribution' instead of `law,'
but the former won't do here for obvious reasons. Despite the overloaded terminology, we will say that two random distributions are \emph{equidistributed} if they have the same law.

Given a Banach space $\scrX$, an $\scrX$-valued random variable is a $(\calF,\operatorname{Borel}(\scrX))$-measurable map $\Omega\to \scrX$. We only consider the case of separable $\scrX$ with separable dual in this paper. If $\Phi$ is an $\scrX$-valued random variable, then $\omega\mapsto \lVert \Phi(\omega) \rVert_{\scrX}$ is a nonnegative real-valued random variable. Likewise, $x^*\circ \Phi$ is a complex-valued random variable for each $x^* \in \scrX^*$. 

\begin{lemma} \label{lem:bof}
	Let $\scrX$ be an arbitrary separable Banach space. 
	Let $\scrX_{\mathrm{weak}}$ denote $\scrX$ endowed with the weak topology, $S\subseteq \scrX^*$ denote a (operator norm) dense set of functionals, and let  $\scrX_S$ denote $\scrX$ endowed with the corresponding weak topology. 
	
	Then, the three Borel $\sigma$-algebras 
	$\operatorname{Borel}(\scrX)$, $\operatorname{Borel}(\scrX_{\mathrm{weak}})$, $\operatorname{Borel}(\scrX_S)$ all agree:
	\begin{equation}
	\operatorname{Borel}(\scrX) = \operatorname{Borel}(\scrX_{\mathrm{weak}}) =  \operatorname{Borel}(\scrX_S).
	\label{eq:kk}
	\end{equation}
	Moreover,
	\begin{equation} 
	\operatorname{Borel}(\scrX_S)=\sigma(\Lambda \in S) 
	\label{eq:jj}
	\end{equation} 
	is the $\sigma$-algebra generated by the elements of $S$. (The same result holds as long as $S$ contains a countable subset $S_0$ such that $\lVert x \rVert_\scrX = \sup_{\Lambda \in S_0} |\Lambda x|$ for all $x \in \scrX$.)
\end{lemma} 
Recall that any operator norm dense set of functionals on a separable Banach space admits a countable subset with the property stated in parentheses above. See \cite[Lemma 6.7]{Carothers}.  (This is a consequence of the Hahn-Banach theorem.) Given such a countable subset, its elements all have operator norm at most one. One key technicality that we must confront below is that the natural notion in which the series below converge is in the topology of $\scrD'(M)$, which is much weaker than the topology of the function spaces we want to find the noises in. Results such as \Cref{lem:bof} will aid us in the deduction of strong convergence from convergence in these weak topologies. See \S\ref{sec:sazonov}. 
\begin{proof}
	Clearly, \cref{eq:kk} holds with `$\supseteq$' in place of `$=$,' $\operatorname{Borel}(\scrX) \supseteq \operatorname{Borel}(\scrX_{\mathrm{weak}}) \supseteq \operatorname{Borel}(\scrX_S) \supseteq \sigma(\Lambda \in S)$. 
	
	To see the converse, it suffices to note that $\operatorname{Borel}(\scrX)$ is generated by a collection of subsets of $\scrX$ which already lie in $\sigma(\Lambda \in S)$. The collection of closed balls in $\scrX$ works --- 
	\begin{itemize}
		\item the open balls in $\scrX$ generate $\operatorname{Borel}(\scrX)$ (by separability), and each open ball is a countable union of closed balls, and so the closed balls generate $\operatorname{Borel}(\scrX)$ as a $\sigma$-algebra, and 
		\item these are all convex and closed and therefore (by the Hahn-Banach theorem) weakly closed.
	\end{itemize}
	This proves the first equality in \cref{eq:kk}, and 
	since the closed balls are bounded (along with weakly closed) they are in fact closed in $\scrX_S$. This proves \cref{eq:kk}. 
	
	Given any $S_0\subseteq S$, $\sigma(\Lambda \in S_0) \subseteq \sigma(\Lambda \in S)$. Since $\scrX$ is separable, any operator norm dense set of functionals contains a countable subset $S_0$ such that $\lVert x \rVert_\scrX = \sup_{\Lambda \in S_0} |\Lambda x|$ for all $x$.
	Then the closed unit ball in $\scrX$ centered at the origin can be written as an intersection of countably many shifted half-spaces -- 
	\begin{equation}
	\Lambda^{-1} ((-\infty,N^{-1}+\lVert \Lambda \rVert_{\mathrm{op}}]) = \{x\in \scrX : \Lambda(x) \leq N^{-1}+\lVert \Lambda \rVert_{\mathrm{op}}\},
	\end{equation}
	$N\in \bbN^+$ (if $\scrX$ is a real Banach space, with a similar definition if $\scrX$ is a complex Banach space) -- 
	defined by the members of $S_0$. It is therefore in $\sigma(\Lambda \in S_0)$, hence in $\sigma(\Lambda \in S)$, and the same holds for the other closed balls.
\end{proof}
It's worth mentioning that when $\scrX$ is a separable Hilbert space -- e.g.\ the Sobolev spaces $\scrH^k(M,g)$ -- and $S$ contains an orthonormal basis, then the coincidences \cref{eq:kk}, \cref{eq:jj} of the $\sigma$-algebras above can easily be seen directly from the $\sigma(\Lambda\in S)$-measurability of the norm $\lVert - \rVert_\scrX : X\to \bbR$. Since this includes the application of interest here, the proof above is somewhat overkill, but the generality of the result (and the argument) is philosophically suggestive. We do not consider $L^p$-based Sobolev regularity here, but if we were to (as in \cite{SchwartzGeometry}\cite{Kusuoka}\cite{Veraar11}) it seems likely the preceding lemma -- in full generality -- would be useful. 

It follows from the previous lemma that, if $\scrX$ is separable, $\scrX$-valued random variables can be added, so the $\scrX$-valued random variables constitute a vector space --- cf.\ \cite{SchwartzRadon}. 
Moreover:

\begin{corollary}
	\label{cor:wl}
	Suppose that $\Phi_0,\Phi_1,\cdots : \Omega\to \scrX$ is a sequence of $\scrX$-valued random variables, where $\scrX$ is separable, $S\subset \scrX^*$ dense, and suppose that $\Phi_N(\omega)$ converges in $\scrX_S$ as $N\to\infty$ for $\bbP$-almost all $\omega\in \Omega$, i.e.\ for all $\omega$ in some $F\in \calF$ with $\bbP(F)=1$. 
	Then, setting
	\begin{equation}
	\Phi_\infty(\omega) = 
	\begin{cases}
	0 & (\omega\not\in F) \\
	\lim \Phi_n(\omega) & (\omega\in F),
	\end{cases} 
	\end{equation}
	where the limit is taken in $\scrX_S$, $\Phi_\infty:\Omega\to \scrX$ is a well-defined $\scrX$-valued random variable.
\end{corollary}

\begin{lemma} 
\label{lem:measurability_lemma} 
 $\operatorname{Borel}(\scrD'(M)) = \sigma(\langle -,\phi_n \rangle_{\scrL^2(M,g)}:n\in \bbN)$. Consequently, a function $\Phi:\Omega\to \scrD'(M)$ is $(\calF,\operatorname{Borel}(\scrD'(M)))$-measurable if and only if 
\begin{equation} 
\langle \Phi(-), \phi_n \rangle_{\scrL^2(M,g)} : \Omega\to \bbC 
\label{eq:9}
\end{equation} 
is  Borel measurable for each $n\in \bbN$. 
\end{lemma} 
\begin{proof} 
The half of the result stating that 
\begin{equation} 
	\operatorname{Borel}(\scrD'(M)) \supseteq \sigma(\langle -,\phi_n \rangle_{\scrL^2(M,g)}:n\in \bbN)
\end{equation} 
is obvious. 

Since, for any $K\in \bbR$, the inclusion $\smash{\scrH^K(M)}\hookrightarrow \scrD'(M)$ is continuous (and therefore measurable), if $S \in \operatorname{Borel}(\scrD'(M))$ then $\scrH^K(M) \cap S$ can be considered as a Borel subset of $\scrH^K(M)$. 
By \Cref{lem:bof},
\begin{equation} 
	\operatorname{Borel}(\scrH^K(M))= \sigma(\langle -,\phi_n \rangle_{\scrL^2(M,g)}|_{\scrH^K(M)}:n\in \bbN).
\end{equation} 
Any element of $\sigma(\langle -,\phi_n \rangle_{\scrL^2(M,g)}|_{\scrH^K(M)}:n\in \bbN)$ has the form $\scrH^K(M)\cap S_0$ for some 
\begin{equation} 
	S_0 \in \sigma(\langle -,\phi_n \rangle_{\scrL^2(M,g)}:n\in \bbN) \subseteq \scrD'(M).
\end{equation} 
So, there exists some such $S_0$ such that $\scrH^K(M)\cap S=\scrH^K(M)\cap S_0$.
Because $\scrH^K(M)=\scrH^K(M,g)$ is in $\sigma(\langle -,\phi_n \rangle_{\scrL^2(M,g)}:n\in \bbN)$ as well, we deduce that 
\begin{equation}
	\scrH^K(M) \cap S \in \sigma(\langle -,\phi_n \rangle_{\scrL^2(M,g)}:n\in \bbN),
\end{equation}
for every $K\in \bbR$. By the Schwartz representation theorem, $\scrD'(M)=\cup_{K\in \bbZ} \scrH^K(M)$, so $S = \cup_{K\in \bbZ} (\scrH^K(M) \cap S) \in \sigma(\langle -,\phi_n \rangle_{\scrL^2(M,g)}:n\in \bbN)$. 
\end{proof} 

Consequently, the sum of two random distributions is a random distribution. (Of course, this can also be proven more directly.) The set of random distributions on $(\Omega,\calF,\bbP)$ is therefore a complex vector space, a fact which we will use without comment below. The analogue of \Cref{cor:wl} holds for random distributions. See \cite[Part II-Chp.\ V]{SchwartzRadon} for a treatment of random vectors (including more general random elements of LCTVSs, not just random elements of Banach and conuclear spaces). We will not be entirely thorough when it comes to spelling out the meaning of terminology which should be interpretable without confusion, but \cite{SchwartzRadon} can be referred to for details (if needed).

Note that if $\scrX\subseteq \scrD'(M)$ is a continuously embedded function space, then an $\scrX$-valued random variable can naturally be considered as a random distribution. In fact the converse is true for the function spaces we consider (namely the $L^2$-based Sobolev spaces, but this also applies to $L^p$-based spaces for $p\in (1,\infty)$, as well as some other more exotic spaces -- e.g. the Besov \cite{Veraar11} or Triebel-Lizorkin spaces -- that arise in functional and harmonic analysis):
\begin{corollary}	
\label{cor:version}
Suppose that we are given a random distribution $\Phi:\Omega \to \scrD'(M)$ and a Banach space $\scrX$ with $\scrD(M)\subseteq \scrX \subseteq \scrD'(M)$, where the inclusions are continuous (which in our terminological usage is what it means for $\scrX$ to be a function space) and suppose further that $\scrD(M)\subseteq \scrX$ is dense, so that we have a natural identification of $\scrX^*$ with a function space:
\begin{equation} 
\scrD(M)\subseteq \scrX^* \subseteq \scrD'(M),
\end{equation} 
with the inclusions continuous. Suppose further that $\scrD(M)\subseteq \scrX^*$ is dense in $\scrX^*$ (in the operator norm topology).
If it is the case that $\Phi(\omega) \in \scrX$ for $\bbP$-almost all $\omega\in\Omega$, then there exists a full measure subset $F\in \calF$ and an $\scrX$-valued random variable 
\begin{equation} \Phi_\scrX:\Omega\to \scrX  
\label{eq:xx}
\end{equation} 
(a ``version'' of $\Phi$) such that $\Phi_{\scrX}(\omega) = \Phi(\omega)$ for all $\omega\in F$. 
\end{corollary}
We will use this result without comment.
\begin{remark} 
We now say a word about regularity in the measure-theoretic sense. 
Since $\scrD'(M) = \cup_{s\in \bbZ} \scrH^s(M)$ is a countable union of the continuous images of the Polish spaces $\scrH^s(M)$ under the embeddings $\scrH^s(M)\hookrightarrow \scrD'(M)$, every Borel probability measure on $\scrD'(M)$ is Radon. In particular, $\Phi_* \bbP$ is Radon. In other words, $\scrD'(M)$ is a Suslin space, and every Borel probability measure on every Suslin space is Radon. Cf.\ \cite[Part I-Chapter II]{SchwartzRadon}. 
\end{remark} 
Suppose we are given a sequence 
\begin{equation} 
\gamma_0,\gamma_1,\gamma_2 , \gamma_3,\cdots : \Omega\to \bbR
\end{equation}  
of independent Gaussian random variables, $\gamma_n \sim \operatorname{N}(\mu_n,\sigma_n)$, whose means $\mu_n \in \bbR$ and standard deviations $\sigma_n \in [0,\infty)$. Here $\operatorname{N}(\mu_n,\sigma_n) : \operatorname{Borel}(\bbR)\to [0,1]$ is the law of a Gaussian with mean $\mu_n$ and standard deviation $\sigma_n$, and `$\sim$' means equidistributed. The Radon-Nikodym derivative of $\operatorname{N}(\mu_n,\sigma_n)$ is given by $\mathrm{d} \mathrm{N}(\mu_n,\sigma_n)/\mathrm{d} \lambda^1 (\gamma) =  (1/\sigma_n) (2\pi)^{-1/2}\exp(- (\gamma-\mu_n)^2 / (2\sigma_n^2))$ in the case $\sigma_n>0$ and $\mathrm{N}(\mu_n,0) = \delta_{\mu_n}$ in the remaining, ``degenerate,'' case. 
We suppose that $\mu_n,\sigma_n$ grow at most polynomially in $n$. 
This latter locution means that there exist some $\{\gamma_n\}_{n\in \bbN}$-dependent constants  $C >0,p\in \bbR$ such that 
\begin{equation}
|\mu_n| , \sigma_n \leq C (n+1)^p 
\label{eq:parameters_upper_bound} 
\end{equation} 
for all $n\in \bbN$. 
We consider, for each $\omega \in \Omega$ and $N\in \bbN$, the partial series 
\begin{equation}
\Phi_N(\omega) = \sum_{n=0}^N \gamma_n(\omega) \phi_n  \in \scrD(M). 
\label{eq:Psi_partial}
\end{equation}
Each of these is a random distribution. Then, $\Phi_N$ converges in $\scrD'(M)$ almost surely as $N\to\infty$:
\begin{lemma}
	\label{global_regularity}
	Under the assumption \cref{eq:parameters_upper_bound}, for $\bbP$-almost all $\omega \in \Omega$ there exists some distribution $\Phi(\omega) \in \cap_{\varepsilon>0} \scrH^{-d(p+1/2+\varepsilon)}(M)$ such that   
	\begin{equation} 
	\Phi_N(\omega)\to \Phi(\omega)
	\label{eq:convergence}
	\end{equation} 
	 in $\scrH^{-d(p+1/2+\varepsilon)}(M)$ as $N\to\infty$ for each $\varepsilon>0$.
\end{lemma}
\begin{proof}
	By the Borel-Cantelli lemma, for each $\varepsilon>0$, $\limsup_{n\to\infty} n^{-p-\varepsilon}|\gamma_n(\omega)| <\infty $ for $\bbP$-almost all $\omega \in \Omega$. This implies that $\{\gamma_n(\omega)\}_{n=0}^\infty \in h^\varsigma(\bbN)$ -- for some $\varsigma$ independent of $\omega$ --  $\bbP$-almost surely. In particular, we can take $\varsigma = -p-2\varepsilon-1/2$. So for almost all $\omega$, the cutoff sequences $\{\gamma_{n,N}(\omega)\}_{n=0}^\infty$ defined by 
	\begin{equation}
	\gamma_{n,N}(\omega) = 
	\begin{cases}
	\gamma_n & (n\leq N) \\
	0 & (n>N) 
	\end{cases}
	\end{equation}
	satisfy $\{\gamma_{n,N}(\omega)\}_{n=0}^\infty \to \{\gamma_n(\omega)\}_{n=0}^\infty$ in $h^\varsigma(\bbN)$ as $N\to\infty$. 
	By the continuity of $\Sigma:h^\varsigma(\bbN)\to \scrH^s(M)$ for $s=d\varsigma$, this implies that $\Phi_N(\omega) = \Sigma \{\gamma_{n,N}(\omega)\}_{n=0}^\infty  $ converges to $\Sigma \{\gamma_n(\omega)\}_{n=0}^\infty \in \scrH^s(M)$ for almost all $\omega \in \Omega$.
\end{proof}
Technically, $\Phi(\omega)$ is a priori only defined for $\bbP$-almost all $\omega$, but going forwards we can modify our probability space so that \cref{eq:convergence} holds for all $\omega \in \Omega$ and all $\varepsilon>0$, in which case $\Phi$ is a well-defined function $\Phi:\Omega\to \scrD'(M)$, 
\begin{equation}
\Phi(\omega) = \sum_{n=0}^\infty \gamma_n(\omega) \phi_n = \lim_{N\to\infty} \sum_{n=0}^N \gamma_n(\omega)\phi_n \in \scrD'(M),
\label{eq:Phi}
\end{equation}   
and moreover $\Phi:\Omega \to \scrH^s(M)$ for $s$ as in the lemma. Since for any $\varphi \in \scrD(M)$, $\langle \Phi ,\varphi \rangle_{\scrL^2} = \lim_{N\to\infty} \langle \Phi_N, \varphi \rangle_{\scrL^2}$ is a limit of measurable functions, it is measurable, and therefore by the criterion \Cref{lem:measurability_lemma}, \cref{eq:Phi} defines a random distribution.  

If in addition to the upper estimate \cref{eq:parameters_upper_bound} we have a lower estimate, then it can be checked that we have a guaranteed amount of irregularity. More precisely:
\begin{proposition}
	Suppose that $\max\{|\mu_n|,\sigma_n\} \geq c (1+n)^q$
	for some $c>0$ and $q\in \bbR$ and a positive density subset of $n\in \bbN$, meaning that 
	\begin{equation} 
	\limsup_{N\to \infty} \frac{1}{N}\# \{n\leq N : \max\{|\mu_n|,\sigma_n\} \geq c (1+n)^q\}>0.
	\label{eq:ls}
	\end{equation} 
	Alternatively, suppose that $\{\mu_n\}_{n=0}^\infty \notin h^{-q-1/2}(\bbN)$. 
	In either case,  
	\begin{equation}	
	\Phi(\omega) \notin \scrH^{-d(q+1/2)}(M) 
	\label{eq:n}
	\end{equation}
for $\bbP$-almost all $\omega \in \Omega$. 
\end{proposition}
\begin{proof} 
	Clearly, $\{\omega\in \Omega : \Phi(\omega) \in \scrH^{-d(q+1/2)}(M)\}$ is in the tail $\sigma$-algebra $\cap_{N \in \bbN} \sigma(\gamma_N,\gamma_{N+1},\cdots)$. By the Kolmogorov zero-one law, this event either has probability zero or probability one.
	
	Suppose that $\Phi(\omega) \in \scrH^{-d(q+1/2)}(M)$ for $\omega$ in a set of positive probability, so that it occurs almost surely. 
	We can consider the square of our probability space and define two independent random distributions 
	\begin{equation} 
	\Phi_{\mathrm{L}}:\Omega^2\to \scrD'(M), \Phi_{\mathrm{R}}:\Omega^2\to \scrD'(M),
	\end{equation} where each of $\Phi_{\mathrm{L}},\Phi_{\mathrm{R}}$ is equidistributed with $\Phi$ and we use L/R to denote the left/right factors of the square.
	(So, letting $\pi_{\mathrm{L}}:\Omega^2\to \Omega$ and $\pi_{\mathrm{R}}:\Omega^2\to \Omega$ denote projection onto the left/right factors respectively, $\Phi_{\mathrm{L}} = \Phi \circ \pi_{\mathrm{L}}$ and $\Phi_{\mathrm{R}} = \Phi\circ \pi_{\mathrm{R}}$. )
	
	Clearly, given our supposition, the independent random distributions $\Phi_{\mathrm{L}},\Phi_{\mathrm{R}}$ are both in $\scrH^{-d(q+1/2)}(M)$ almost surely, and so their difference $\Phi_{\mathrm{L}}-\Phi_{\mathrm{R}}$ lies in $\scrH^{-d(q+1/2)}(M)$ almost surely as well. Note that $\Phi_{\mathrm{L}}-\Phi_{\mathrm{R}}$ is equidistributed with 
	\begin{equation}
	\sqrt{2} \Big(\Phi - \sum_{n=0}^\infty \mu_n \phi_n  \Big)\sim \sum_{n=0}^\infty \sqrt{2} \sigma_n \tilde{\gamma}_n \phi_n,  
	\label{eq:us}
	\end{equation}
	where $\tilde{\gamma}_0,\tilde{\gamma}_1,\tilde{\gamma}_2,\ldots$ are arbitrary i.i.d.\ standard Gaussian random variables on some probability space. Since the left-hand side and the first term of the left-hand side of \cref{eq:us} are together in $\scrH^{-d(q+1/2)}(M)$ almost surely, we conclude that the second term on the left-hand side is as well. By equidistribution, the right-hand side of \cref{eq:us} is in $\scrH^{-d(q+1/2)}(M)$ almost surely. To summarize, if $\Phi \in \scrH^{-d(q+1/2)}(M)$ with positive probability, then 
	\begin{equation}
	\sum_{n=0}^\infty \mu_n \phi_n \in \scrH^{-d(q+1/2)}(M) \quad \text{ and } \quad \sum_{n=0}^\infty \sigma_n \tilde{\gamma}_n \phi_n \in \scrH^{-d(q+1/2)}(M)
	\label{eq:z}
	\end{equation}
	with positive probability. 
	If the hypothesis \cref{eq:ls} of the proposition holds, then either 
	\begin{itemize} \item $|\mu_n| \geq c(1+n)^q$ for a positive density subset of $n\in \bbN$ or 
	\item  
		$\;\sigma_n\; \geq c(1+n)^q$ for a positive density subset of $n$. 
	\end{itemize} The former is clearly inconsistent with the first condition in \cref{eq:z}, and the latter can be shown to be inconsistent with the second condition in a few different ways, e.g.\ using Kolmogorov's three series theorem, cf.\ \cite[Chapter 22, Theorem 22.8]{Billingsley}. Likewise, if $\{\mu_n\}_{n=0}^\infty \notin h^{-q-1/2}(\bbN)$, then $\Sigma \{\mu_n\}_{n=0}^\infty \notin \scrH^{-d(q+1/2)}(M)$, which contradicts the first condition in \cref{eq:z}. 
\end{proof}

Combining the previous propositions:

\begin{propositionp}
	\label{global_regularity_sharp} 
	If the standard deviations $\sigma_n$ and means $\mu_n$ of our random Gaussian coefficients $\gamma_n$ satisfy the conditions 
	\begin{equation} 
	0<\liminf_{n\to\infty} (1+n)^{\varsigma}\sigma_n \leq \sup_{n\in \bbN} (1+n)^{\varsigma}  \sigma_n < \infty 
	\end{equation}
	and $\{\mu_n\}_{n=0}^\infty \in h^{\varsigma-1/2}(\bbN)$, 
	then 
	for $\bbP$-almost all $\omega \in \Omega$ the partial series $\Phi_N(\omega)$ converges as $N\to\infty$ to some 
	\begin{equation}
	\Phi(\omega) \in \cap_{\varepsilon>0} \scrH^{d(\varsigma-1/2-\varepsilon)}(M) \backslash \scrH^{d(\varsigma-1/2)}(M) 
	\end{equation}
	in the topology of $\scrH^{d(\varsigma-1/2-\varepsilon)}(M,g)$ for any $\varepsilon>0$. 
	$\Phi$ is a random distribution, and in particular by \Cref{cor:version} admits a $\smash{\scrH^{d(\varsigma-1/2-\varepsilon)}(M,g)}$-valued version. 
\end{propositionp}

The preceding argument does not yield negative results regarding local regularity because given a bump function $\chi \in \scrD(M)$, the functions $
\chi \phi_0, \chi \phi_1, \chi \phi_2, \ldots $
are no longer typically orthogonal elements of $\scrL^2(M,g)$. We need a more sophisticated analysis. The full local and microlocal result, \Cref{thm:main}, given in \S\ref{sec:main_theorem}, refines \Cref{global_regularity_sharp}.

\section{A consequence of Sazonov's theorem}
\label{sec:sazonov}
In this section, we assume that $\mu_n = 0$, so that $\gamma_0,\gamma_1,\gamma_2,\cdots:\Omega\to \bbR$ are independent centered Gaussian random variables. 
Given a fixed sequence ${\bm\sigma}=\{\sigma_n\}_{n=0}^\infty$ of $\sigma_n>0$, let $h_{\bm\sigma}(\bbN)$ denote the Hilbert space which as a set is given by 
\begin{equation}
h_{\bm\sigma} (\bbN) = \{ \{a_n\}_{n=0}^\infty \subset \bbC : \sum_{n=0}^\infty \sigma_n^{-2} |a_n|^2  < \infty \} 
\end{equation}
and whose inner product $\langle -,- \rangle_{h_{\bmsigma}} : h_{\bmsigma}(\bbN)\times h_{\bmsigma}(\bbN)\to \bbC$ is given by $\langle \{a_n\}_{n=0}^\infty , \{b_n\}_{n=0}^\infty \rangle_{h_{\bmsigma}} = \sum_{n=0}^\infty |\sigma_n|^{-2} a_n^* b_n$.
The degenerate case $\sigma_n=0$ is, while unproblematic, a bit technically awkward, hence our temporary assumption to the contrary. (Thus, the results of this section will take a small modification to apply to the massless Gaussian free field.) 
If $\bmsigma = \{(1+n)^\varsigma\}_{n=0}^\infty$  for some  $\varsigma \in \bbR$, $h_{\bmsigma}(\bbN)=h^{-\varsigma}(\bbN)$. 
If $\sigma_n$ grows at most polynomially in $n$, then $h_{\bmsigma}(\bbN)\subset d'(\bbN)$, and the map $\Sigma$ defined by \cref{eq:generic_sum} restricts to a continuous embedding 
\begin{equation}
\Sigma|_{h_{\bmsigma}(\bbN)} : h_{\bmsigma}(\bbN)\to \scrD'(M),
\label{eq:Sigma_restricted} 
\end{equation}
and as in the previous section we can replace `$\scrD'(M)$' on the right-hand side with some $\{\sigma_n\}_{n=0}^\infty$-dependent $L^2$-based Sobolev space.

We want to know for which $s\in \bbR$ and ``microlocal cutoffs'' $\operatorname{Op}(\chi)$ -- $\chi$ being an appropriate function on the cosphere bundle or symbol on the punctured cotangent bundle  and `$\operatorname{Op}(\chi)$' denoting an appropriate quantization thereof (see \S\ref{sec:main_theorem}) --  $ \operatorname{Op}(\chi)(1+\triangle_g)^s \Phi(\omega) \notin \scrL^2(M)$ holds for $\bbP$-almost all $\omega \in \Omega$. Our arguments in this section apply to general continuous linear operators $L:\scrD'(M)\to \scrD'(M)$ in place of $\operatorname{Op}(\chi)(1+\triangle_g)^s $, including all Kohn-Nirenberg $\Psi$DOs. 
The main result of this section is the following consequence (or perhaps variant) of Sazonov's theorem (together with its converse). (See \cite{SchwartzRadon} for similar statements. Our formulation differs somewhat from the formulation there, but this difference is minor. Hence we do not claim any novelty here, except in presentation.)

\begin{proposition}\label{Sazonov}
	Suppose that $L:\scrD'(M)\to \scrD'(M)$ is a continuous linear operator and that $\bmsigma=\{\sigma_n\}_{n=0}^\infty$ grows at most polynomially in $n$. 
	Then the random distribution \cref{eq:Phi} satisfies 
	\begin{itemize}
		\item[(I)]  $L \Phi(\omega) \in \scrL^2(M)$ for $\bbP$-almost all $\omega \in \Omega$   
	\end{itemize} 
	if and only if 
	\begin{itemize}
		\item[(II)]$L\circ \Sigma(h_{\bmsigma}(\bbN))\subseteq \scrL^2(M)$ and   
		\begin{equation}
		L\circ \Sigma|_{h_{\bmsigma}(\bbN)} : h_{\bmsigma}(\bbN)\to \scrL^2(M,g), 
		\label{eq:sigma_restricted} 
		\end{equation}   is Hilbert-Schmidt.
	\end{itemize}
Note that (II) is equivalent to the conjunction of having the inclusion $L \phi_n \in \scrL^2(M)$ for all $n\in \bbN$ and having the estimate $\sum_{n=0}^\infty \sigma_n^2 \lVert L \phi_n \rVert^2_{\scrL^2(M,g)}<\infty$.
\end{proposition}
Recall that a linear mapping $L:\scrX\to \scrY$ of separable infinite dimensional Hilbert spaces $\scrX,\scrY$ is called \emph{Hilbert-Schmidt} if it satisfies 
\begin{equation}
\lVert L \rVert_{\mathrm{HS}}^2 \overset{\mathrm{def}}{=} \sum_{n=0}^\infty \lVert L x_n \rVert_{\scrY}^2= \operatorname{Tr} L^*L  < \infty
\label{eq:HS_condition} 
\end{equation}
for some orthonormal basis $\{x_n\}_{n=0}^\infty$ of $\scrX$, in which case $L$ satisfies \cref{eq:HS_condition} \emph{for all} orthonormal bases $\{x_n\}_{n=0}^\infty$ of $\scrX$. A Hilbert-Schmidt mapping is necessarily bounded and moreover compact. Given an orthonormal basis $\{x_n\}_{n=0}^\infty\subseteq \scrX$ and arbitrary $\{y_n\}_{n=0}^\infty \subseteq \scrY$, if a mapping $\{x_n\}_{n=0}^\infty \to \scrY$, $x_n\mapsto y_n$, extends to a 
bounded linear operator $\scrX\to \scrY$, the resultant extension is unique and is Hilbert-Schmidt if and only if 
\begin{equation} 
\sum_{n=0}^\infty \lVert y_n \rVert^2_\scrY < \infty.
\label{eq:m5}
\end{equation} 
The estimate \cref{eq:m5} implies the extension to a bounded linear operator $L:\scrX\to \scrY$ satisfying $\lVert L \rVert_{\mathrm{op}} \leq \lVert L \rVert_{\mathrm{HS}} = \smash{(\sum_{n=0}^\infty \lVert y_n \rVert_{\scrY}^2 )^{1/2}}$, and so \cref{eq:m5} is a necessary and sufficient condition for the extension of the given mapping $\{x_n\}_{n=0}^\infty \to \scrY$ to a Hilbert-Schmidt operator. 
This proves the equivalence of the formulation of (II) given at the end of \Cref{Sazonov} with (II) itself. Note that the Hilbert-Schmidt operators constitute an operator ideal (see e.g.\ \cite[Theorem 3.8.2]{Simon2015operator}). We refer the reader to \cite[Chapter 3, \S8]{Simon2015operator} for an exposition of the theory of Hilbert-Schmidt operators.

The whole proposition follows in short order from the following conjunction of mostly well-known results. Suppose that $\scrX$ is a separable $\bbC$-Hilbert space, $\{x_n\}_{n=0}^\infty \subseteq \scrX$, and $(\tilde{\Omega},\tilde{\calF},\tilde{\bbP})$ is a probability space on which i.i.d. standard Gaussian random variables $\tilde{\gamma}_0,\tilde{\gamma}_1,\tilde{\gamma}_2,\ldots:\tilde{\Omega}\to \bbR$
are defined.    
In the following, given (not necessarily orthonormal) $x_0,x_1,x_2,\ldots \in \scrX$, we set  
\begin{equation} 
\Sigma_N = \Sigma_N(-)=\sum_{n=0}^N \tilde{\gamma}_n(-) x_n. 
\label{eq:sn}
\end{equation}
Each of these is an $\scrX$-valued random variable, in the sense of \S\ref{subsec:random}.   Recall that the sequence $\{\Sigma_N\}_{N=0}^\infty$ is said to converge (strongly) \emph{in probability} to another $\scrX$-valued random variable $\Sigma$ if 
\begin{equation} 
\lim_{N\to\infty}\tilde{\bbP}[ \lVert \Sigma_N - \Sigma \rVert_\scrX>\delta]\to 0 
\end{equation} 
for each $\delta>0$, i.e.\ if the $\bbR$-valued random variable $\lVert \Sigma_N - \Sigma \rVert_{\scrX}$ converges to zero in probability.
Analogously, we say that $\Sigma_N$ converges \emph{weakly in probability} to $\Sigma$ if 
$\langle \Sigma_N , x^* \rangle_\scrX \to \langle \Sigma,x^* \rangle_\scrX$ as $N\to\infty$ in probability for each $x^* \in \scrX$.  (See \cite{Billingsley} for the $\bbR$-valued case, \cite{SchwartzRadon}\cite{Hytonen2016analysis} for the generalization to Banach spaces.) In addition, let $L^p(\tilde{\Omega},\tilde{\calF},\tilde{\bbP};\scrX)$ denote for each $p\in [1,\infty)$ the Banach space of $p$-integrable $X$-valued random variables on the given probability space --- see \cite{Hytonen2016analysis}. 
\begin{proposition} \label{prop:Sazonov_equiv}
Let $\scrX$ be a separable Hilbert space. 
Given the setup above, with $\Sigma_N$ defined by \cref{eq:sn}, the following seven conditions are equivalent.
	\begin{enumerate}
		\item $\sum_{n=0}^\infty \lVert x_n \rVert_\scrX^2 < \infty$, 
		\item $\sum_{n=0}^N \tilde{\gamma}_n(\tilde{\omega}) x_n$ converges \;weakly\; in $\scrX$ as $N\to\infty$ for $\tilde{\bbP}$-almost all $\tilde{\omega} \in \tilde{\Omega}$, 
		\label{item:convergence_i}
		\item $\sum_{n=0}^N \tilde{\gamma}_n(\tilde{\omega}) x_n$ converges strongly in $\scrX$ as $N\to\infty$ for $\tilde{\bbP}$-almost all $\tilde{\omega}\in \tilde{\Omega}$. 
		\label{item:convergence_ii}
		\item $\sum_{n=0}^N \tilde{\gamma}_n(-) x_n$ converges \;weakly\; in $\scrX$ as $N\to\infty$ in probability. 
		\item $\sum_{n=0}^N \tilde{\gamma}_n(-) x_n$ converges strongly in $\scrX$ as $N\to\infty$ in probability. 
		\item $\sum_{n=0}^N \tilde{\gamma}_n(-) x_n$ converges in $L^p(\tilde{\Omega},\tilde{\calF},\tilde{\bbP};\scrX)$ as $N\to\infty$ for some $\,p\in [1,\infty)$,  i.e.\ that 
		\begin{equation}
			\lim_{N\to\infty} \sup_{N' \geq N} \int_{\tilde{\Omega}} \Big\lVert \sum_{n=N}^{N'} \tilde{\gamma}_n(\tilde{\omega}) x_n \Big\rVert_{\scrX}^p \dd \tilde{\bbP} (\tilde{\omega}) =0 
			\label{eq:convinlpdef}
		\end{equation}
		for some $p\in [1,\infty)$. 
		\label{item:convergence_iii}
		\item $\sum_{n=0}^N \tilde{\gamma}_n(-) x_n$ converges in $L^p(\tilde{\Omega},\tilde{\calF},\tilde{\bbP};\scrX)$ as $N\to\infty$ for every $p\in [1,\infty)$, i.e.\ that \cref{eq:convinlpdef} holds for all $p\in [1,\infty)$.	
		\label{item:convergence_iv}
	\end{enumerate}
Moreover, it suffices to verify (2), (4) on a dense set of functionals. This means (by Riesz duality) that the following condition is equivalent to all of those above: there exists a total subset $S\subseteq \scrX$ and an $\scrX$-valued random variable $\Sigma:\smash{\tilde{\Omega}}\to \scrX$ such that \begin{equation} 
\langle \Sigma_N, x^* \rangle_\scrX\to \langle \Sigma,x^* \rangle_\scrX
\end{equation} 
in probability whenever $x^* \in S$.

If the equivalent conditions above hold, then the well-defined law $\Sigma_* \tilde{\bbP}:\operatorname{Borel}(\scrX)\to [0,1]$
of the almost surely existing strong limit 
\begin{equation}
\Sigma=\lim_{N\to\infty} \sum_{n=0}^N \tilde{\gamma}_n x_n
\end{equation} 
(which makes sense as an $\scrX$-valued random variable) is a centered Gaussian.  

\end{proposition}

The final locution in the proposition just means  that $\smash{\langle \Sigma,x \rangle_\scrX: \tilde{\Omega}\to \bbR}$ is a (possibly degenerate) centered Gaussian random variable for each $x\in \scrX$.
Recall that a \emph{total} subset of a Banach space is a set whose algebraic span is dense. 

\begin{remark}
	A priori, the set of $\omega\in \Omega$ for which the formal series $\sum_{n=0}^\infty \tilde{\gamma}_n(\omega)x_n$ is weakly summable (resp.\ strongly summable) is in $\sigma(\tilde{\gamma}_0,\tilde{\gamma}_1,\tilde{\gamma}_2,\ldots)\subseteq \calF$. 
	\begin{itemize}
		\item 	This is clear in the case of strong summability, since the differences $\lVert \Sigma_N-\Sigma_M \rVert_\scrX^2$ are all measurable with respect to the given $\sigma$-algebra, and $\Sigma_N(\omega)$ converges strongly as $N\to\infty$ if and only if the sequence $\{\Sigma_N(\omega)\}_{N=0}^\infty$ is Cauchy.
		\item 	In the case of weak summability, note that weak summability implies strong boundedness (by the uniform boundedness principle --- cf.\ \cite[Exercise 3.16]{RS}) of partial sums, and for bounded sequences weak convergence may be verified by testing against a countable operator norm dense subset $S \subset \scrX^*$ of functionals. Let $\scrX_0$ denote a countable dense subset of $\scrX$, and consider the subset $F\subset \Omega$ of $\omega\in \Omega$ such that $\sup_{N} \lVert \Sigma_N(\omega) \rVert_\scrX < \infty$ and that  for each rational $\varepsilon>0$ there exists some $\Sigma_{\mathrm{approx}} \in \scrX_0$ such that for all $\Lambda \in S \backslash \{0\}$, there exists some $M=M(\Lambda,\Sigma_{\mathrm{approx}},\omega)$ such that 
		\begin{equation} 
		|\langle \Lambda, \Sigma_N-\Sigma_{\mathrm{approx}} \rangle| < \lVert \Lambda \rVert_{\mathrm{op}} \varepsilon
		\end{equation}
		for all $N\geq M$. Then,
		$F \in \sigma(\tilde{\gamma}_0,\tilde{\gamma}_1,\tilde{\gamma}_2,\ldots)$, and  $\Sigma_N(\omega)$ converges weakly as $N\to\infty$ if and only if $\omega\in F$.
	\end{itemize}		
	In fact, this set is in the tail $\sigma$-algebra $\cap_{N>0} \sigma(\tilde{\gamma}_N,\tilde{\gamma}_{N+1},\tilde{\gamma}_{N+2},\ldots)$, so by the Kolmogorov zero-one law either $\Sigma_N$ converges weakly (resp.\ strongly) almost surely or it fails to do so almost surely. 
\end{remark}

The fact that it suffices to verify (4) on a dense set of functionals can be seen from a standard proof of (4) $\Rightarrow$ (1), which can be modified to use only a dense set of functionals (as essentially noted in \cite{Hoffmann1974}). We will use the continuity of the Fourier transform of an $\scrX$-valued random variable with respect to the strong topology  \cite[Volume II pg. 122, alternatively 7.3.16]{Bogachev}\cite[Part II- Chapter 2, Theorem 1]{SchwartzRadon} in order to make the needed inference. Cf.\ \Cref{Ito-Nisio}, \Cref{density_lemma}.

We only need the equivalence of (1) and (2) -- with (2) weakened as noted above -- to prove \Cref{Sazonov}, but the equivalence of the other conditions is interesting in its own right and useful in proving the needed result, so we take the time to discuss it here. See \cite[Theorem 6.4.1, Corollary 6.4.4]{Hytonen2016analysis} for a textbook presentation of some of this.

The upshot of this section will be three independent proofs of the important implication (2) $\Rightarrow$ (1), where (2) is to be verified on dense set of functionals: 
one in terms of Sazonov's theorem and its converse (which gave this section its title), one in terms of Fernique's theorem \cite{Fernique}, and one in terms of the It\^o-Nisio theorem \cite{ItoNisio} in conjunction with the Paley-Zygmund inequality and Kahane-Khintchine inequalities. Somewhat arbitrarily, we split the proofs among two subsections, \S\ref{subsec:3.f} and \S\ref{subsec:3.st}.

We now turn to showing that \Cref{prop:Sazonov_equiv} implies \Cref{Sazonov}. 
First, we check (though it is not needed for the proof of the main theorem) that (I) can only hold if it is the case that $L \phi_n \in \scrL^2(M)$ for each $n\in \bbN$. 

\begin{lemma} \label{lem:auxiliary}
	If $L\Phi \in \scrL^2(M)$ with positive probability, meaning that the set of $\omega\in \Omega$ for which $L\Phi(\omega) \in \scrL^2(M)$ is in $\calF$ and satisfies $\bbP[L\Phi \in \scrL^2(M)]>0$, then $L\phi_n \in \scrL^2(M)$ for each $n\in \bbN$. 
\end{lemma} 
Note that if we do not already know the conclusion of the lemma, it is not a priori obvious that 
\begin{equation} 
(L\circ \Phi)^{-1}(\scrL^2(M))=\{\omega\in \Omega: L\Phi(\omega) \in \scrL^2(M)\} 
\end{equation} 
 is an event in the tail $\sigma$-algebra $\cap_{N>0} \sigma(\gamma_N,\gamma_{N+1},\gamma_{N+2},\ldots)$, and so we cannot directly apply the Kolmogorov zero-one law to deduce that 
\begin{equation}
\bbP[L\Phi \in \scrL^2(M)] > 0 \iff \bbP[L \Phi \in \scrL^2(M)] = 1.
\label{eq:0}
\end{equation}
If, however, we do know that $L\phi_n \in \scrL^2(M)$ for each $n\in \bbN$, then the zero-one law applies and \cref{eq:0} follows. Let's begin by proving \Cref{lem:auxiliary} under the additional assumption (I), i.e.\ that the right-hand side of \cref{eq:0} holds. 

\begin{proof}[Proof of \Cref{lem:auxiliary} if (I) holds]
	Fix $n\in \bbN$. 
	Note that the two random distributions
	\begin{equation}
	L \Phi = \sum_{m=0}^\infty \gamma_m L \phi_m  \quad \text{ and }\quad - \gamma_n L\phi_n +\sum_{m=0, m\neq n}^\infty \gamma_m L \phi_m 
	\end{equation}
	are equidistributed, so if condition (I) holds then both random distributions lie in $\scrL^2(M)$ almost surely. Given this, their difference 
	\begin{equation} 
	2\gamma_n L \phi_n  = L\Phi - \Big(  - \gamma_n L\phi_n +\sum_{m=0, m\neq n}^\infty \gamma_m L \phi_m   \Big) 
	\end{equation} 
	lies in $\scrL^2(M)$ almost surely. So, $L\phi_n \in \scrL^2(M)$.
\end{proof}
Using a union bound, the previous argument works as long as $\bbP[L\Phi \in \scrL^2(M)] >1/2$, in which case we conclude that $\gamma_n(\omega) L \phi_n \in \scrL^2(M)$ for some positive measure set of $\omega$ and therefore that $L \phi_n \in \scrL^2(M)$.
By promoting the $\gamma_n$ to complex Gaussians $\gamma_{n,\bbC} = \gamma_n + i \bar{\gamma}_n$ -- where $\gamma_0,\bar{\gamma}_0,\gamma_1,\bar{\gamma}_1,\ldots$ are mutually independent i.i.d.\ standard Gaussians -- and considering the equidistributed Gaussian random variables $\smash{e^{ 2\pi i m/ N}\gamma_{n,\bbC}}$ for $m\in \{1,\ldots,N\}$ in place of $\pm \gamma_n$, we can modify the previous argument to apply  as long as 
\begin{equation} 
\bbP[L\Phi \in \scrL^2(M)]> 1/N. 
\label{eq:1} 
\end{equation} 
If $L\Phi \in \scrL^2(M)$ with positive probability, then we can choose $N$ sufficiently large such that the hypothesis \cref{eq:1} holds and then conclude \Cref{lem:auxiliary} in full generality. Alternatively, we can use the following purely measure-theoretic fact:

\begin{lemma}
	If $F \in \sigma(\gamma_0,\gamma_1,\ldots)\subseteq \calF$ has positive measure, then for all $n\in \bbN$ there exist $\omega,\omega' \in F$ such that $\gamma_m(\omega) = \gamma_m(\omega')\iff m\neq n$. 
	\label{measure_lem}
\end{lemma}
\begin{proof} 
	Fix an $n\in \bbN$. 
	By the independence of the $\sigma$-algebras $\sigma(\gamma_n)$ and $\sigma(\gamma_m : m\neq n)$, we can 
	construct probability spaces $(\bbR,\operatorname{Borel}(\bbR),\bbP_n)$ and $(\Omega_n',\sigma(\tilde{\gamma}_{m,n} : m\neq n),\bbP_n')$ and an isomorphism 
	\begin{equation}
	\iota_n : (\bbR,\operatorname{Borel}(\bbR),\bbP_n) \times (\Omega_n',\sigma(\tilde{\gamma}_{m,n} : m\neq n),\bbP_n') \to (\Omega,\sigma(\gamma_0,\gamma_1,\ldots),\bbP)
	\label{eq:in}
	\end{equation}
	such that $\gamma_m \circ \iota_n  = \tilde{\gamma}_{m,n} \circ \pi^{(i)}$ for all $m\in \bbN$, where $i=1$ if $m=n$ and $i=2$ if $m\neq n$, where $\pi^{(1)},\pi^{(2)}$ are the canonical projections from the domain of \cref{eq:in} to the first or second factor respectively. Here $\tilde{\gamma}_{n,n}: \bbR\to \bbR$ is the identity, $\bbP_n$ is the law of $\gamma_n$ (i.e.\ a nondegenerate Gaussian), and $\tilde{\gamma}_{m,n}:\Omega_n'\to \bbR$ for $m\neq n$ are other random variables. By Fubini's theorem, we can write 
	\begin{equation}
	0 < \bbP[ \tilde{F} ] =  \int_{\Omega_n'}\Big(\int_{\bbR} 1_{\tilde{F}} \dd \bbP_n(\tilde{\gamma}_{n,n}) \Big) \dd \bbP_n'(\tilde{\gamma}_{m,n} : m\neq n) 
	\label{eq:m3}
	\end{equation}
	with $\tilde{F} = \iota_n^{-1}(F)\subseteq \bbR\times \Omega_n'$, where the right-hand side is a well-defined iterated integral.
	Since $\bbP_n$ is not supported on a single point (by the nondegeneracy assumption), \cref{eq:m3} implies that there exists some $w' \in \Omega_n'$ and distinct $r,\rho \in \bbR$ such that $(r,w'),(\rho,w') \in \tilde{F}$.  
	Setting $\omega = \iota_n(r,w')$ and $\omega' = \iota_n(\rho,w')$, $\omega,\omega' \in F$, and we see that 
	\begin{equation}
	\gamma_m(\omega) = \gamma_m\circ \iota_n(r,w') = \tilde{\gamma}_{m,n} \circ \pi^{(i)}(r,w') = 
	\begin{cases}
	r & (m=n), \\
	\tilde{\gamma}_{m,n}(w') & (\text{otherwise}), 
	\end{cases}
	\end{equation}
	where $i$ is as above, 
	and likewise for $\omega'$, with $\rho$ in place of $r$,  
	\begin{equation}
	\gamma_m(\omega') = \gamma_m\circ \iota_n(\rho,w') = \tilde{\gamma}_{m,n} \circ \pi^{(i)}(\rho,w') = 
	\begin{cases}
	\rho & (m=n), \\
	\tilde{\gamma}_{m,n}(w') & (\text{otherwise}), 
	\end{cases}
	\end{equation} 
	Therefore $\omega,\omega'$ have the desired properties. 
\end{proof}

\begin{proof}[Proof of \Cref{lem:auxiliary}]
	If (I) holds, let $F\in \calF$ be the positive measure set of all $\omega\in \Omega$ for which $L \Phi (\omega) \in \scrL^2(M)$. By \Cref{measure_lem}, for each $n\in \bbN$ there exists some $\omega,\omega' \in F$ such that 
	\begin{equation}
	\gamma L\phi_n + \sum_{m\neq n} \gamma_m(\omega) L \phi_m  \in \scrL^2(M,g) 
	\label{eq:03}
	\end{equation}
	for each $\gamma \in \{\gamma_n(\omega),\gamma_n(\omega')\}$ and $\gamma_n(\omega)\neq \gamma_n(\omega')$. Taking the difference of the series defined in \cref{eq:03} for both possible values of $\gamma$, we get $(\gamma_n(\omega)-\gamma_n(\omega')) L \phi_n \in \scrL^2(M,g)$, and this implies that $L \phi_n \in \scrL^2(M,g)$. 
\end{proof}

\begin{proof}[Proof of \ref{prop:Sazonov_equiv} $\Rightarrow$ \ref{Sazonov}]
	Given the setup of \Cref{Sazonov}, let $\tilde{\gamma}_n = \sigma_n^{-1}\gamma_n $. Then $\tilde{\gamma}_0,\tilde{\gamma}_1,\tilde{\gamma}_2,\ldots$ are i.i.d.\ standard Gaussian random variables, and we have 
	\begin{equation}
	\sum_{n=0}^N \tilde{\gamma}_n(\omega) L \circ \Sigma|_{h_{\bmsigma}} (\sigma_n \bmdelta_n) =\sum_{n=0}^N \tilde{\gamma}_n(\omega) \sigma_n L \phi_n = \sum_{n=0}^N \gamma_n(\omega) L \phi_n = L \Phi_N(\omega), 
	\label{eq:13}
	\end{equation}
	where $\{\sigma_n \bmdelta_n\}_{n=0}^\infty$ is the standard orthonormal basis for $h_{\bmsigma}(\bbN)$. 
	
	\begin{itemize}
		\item (II)$\Rightarrow$(I): 
		If $L\circ \Sigma|_{h_{\bmsigma}} : h_{\bmsigma}(\bbN)\to \scrL^2(M,g)$ is Hilbert-Schmidt, then \Cref{prop:Sazonov_equiv} applies, and we conclude using \cref{eq:13} that $L\Phi_N(\omega)$ converges ($\scrL^2(M)$-)weakly to some $\Phi_{L,\infty}(\omega) \in \scrL^2(M)$ for $\bbP$-almost all $\omega\in\Omega$.
		For such $\omega$, 
		\begin{equation} 
		\;L\Phi_N(\omega)\to \Phi_{L,\infty}(\omega)\text{ in $\scrD'(M)$}
		\label{eq:t5}
		\end{equation} 
		(since $\scrL^2(M)\hookrightarrow \scrD'(M)$ is continuous). Since by definition $\Phi_N\to \Phi$ in $\scrD'(M)$ (and by assumption $L:\scrD'(M)\to \scrD'(M)$ is continuous), $\Phi_\infty = \Phi$ satisfies 
		\begin{equation} 
		L\Phi_N(\omega)\to L \Phi_\infty(\omega)\;\text{ in $\scrD'(M)$}.
		\label{eq:t6}
		\end{equation} 
		Recall that we modified our probability space so that $\Phi_N(\omega)\to \Phi(\omega)$ in $\scrD'(M)$ for all $\omega\in \Omega$. Since $\scrD'(M)$ is a Hausdorff topological space, limits in it are unique, so the conjunction of \cref{eq:t5} and \cref{eq:t6} implies $\Phi_{L,\infty}(\omega) = L\Phi(\omega)$.
		We conclude that $L\Phi(\omega) \in \scrL^2(M)$ for $\bbP$-almost all $\omega\in \Omega$, as desired. 
		\item (I)$\Rightarrow$(II): we already know that if (I) holds, by \Cref{lem:auxiliary} $L\phi_n \in \scrL^2(M)$ for all $n\in \bbN$, so $L\Phi_N \in \scrL^2(M)$ as well. 
		By the continuity of $L$, $L \Phi_N (\omega) \to L \Phi(\omega)$ in $\scrD'(M)$. 
		In other words, 
		\begin{equation}
		\langle \varphi, L \Phi_N (\omega) \rangle_{\scrL^2(M,g)} \to \langle \varphi, L \Phi(\omega) \rangle_{\scrL^2(M,g)} , 
		\label{eq:r93}
		\end{equation}
		for  all $\varphi \in \scrD(M)$.   
		If (I) holds, then we can consider $L\Phi$ as an $\scrL^2(M,g)$-valued random variable. 
		\Cref{eq:r93} tells us that $L\Phi_N \to L\Phi$ in the weak topology generated by the elements of $\scrD(M)$ acting as linear functionals on $\scrL^2(M,g)$. These functionals are dense in $\scrL^2(M,g)$, so \Cref{prop:Sazonov_equiv} applied to \cref{eq:13} implies 
		\begin{equation}
		\sum_{n=0}^\infty \sigma_n^2 \lVert L \phi_n \rVert_{\scrL^2(M,g)}^2 < \infty .
		\end{equation}
		Therefore $L \circ \Sigma|_{h_{\bmsigma}}(h_{\bmsigma}(\bbN)) \subseteq \scrL^2(M)$, and the map \cref{eq:sigma_restricted} is Hilbert-Schmidt. 
	\end{itemize}
This completes the proof that \Cref{prop:Sazonov_equiv} implies \Cref{Sazonov}. 
\end{proof}

\subsection{First Proof of \Cref{prop:Sazonov_equiv}}
\label{subsec:3.f}
We now turn to the proof of \Cref{prop:Sazonov_equiv}.
In the following, all instances of `(1)', `(2)', `(3)', etcetera refer to the conditions in the proposition. 
The implications (3) $\Rightarrow$ (2), (3) $\Rightarrow$ (5), (5) $\Rightarrow$ (4), (7) $\Rightarrow$ (6), and (6) $\Rightarrow$ (4), (5)  and those that follow from these are all obvious. We only mention the other implications below. 
We will go back and forth between assuming that $\scrX$ is a general separable Banach space and that $\scrX$ is a separable Hilbert(izable) space. If it isn't specified otherwise, then $\scrX$ is assumed to be a Hilbert space, and with regards to the rest of the paper it suffices to assume as such.

\begin{lemma}
	\label{lem:6->1}
	(7) $\Rightarrow$ (1).
\end{lemma}
\begin{proof} 
By the independence of the $\tilde{\gamma}_n$, we see that 
\begin{equation}
\Big\lVert \sum_{n=0}^N \tilde{\gamma}_n(-) x_n \Big\rVert_{L^2(\tilde{\Omega},\tilde{\calF},\tilde{\bbP};\scrX)}^2 = \sum_{n=0}^N \lVert x_n \rVert_\scrX^2
\label{eq:3h}
\end{equation}
for each $N\in \bbN$. If (7) holds, then the left-hand side of \cref{eq:3h} has to converge to something finite as $N\to\infty$, which just means that (1) holds. 
\end{proof}

\begin{lemma}
	(1) $\Rightarrow$ (6).
	\label{lem:1->5}
\end{lemma}
\begin{proof}
	As in \cref{eq:3h}, 
\begin{equation}
\Big\lVert \sum_{n=0}^N \tilde{\gamma}_n(-) x_n - \sum_{n=0}^M \tilde{\gamma}_n(-) x_n \Big\rVert_{L^2(\tilde{\Omega},\tilde{\calF},\tilde{\bbP};\scrX)}^2 = \sum_{n=N+1}^M \lVert x_n \rVert_\scrX^2 
\label{eq:h7}
\end{equation}
for $N\leq M$. If (1) holds, then the right-hand side of \cref{eq:h7} is $o(1)$ as $N\to\infty$, uniformly in $M$. This means that the sequence $\{\Sigma_N \}_{N=0}^\infty \in L^2(\tilde{\Omega},\tilde{\calF},\tilde{\bbP};\scrX)$
is Cauchy in $L^2(\tilde{\Omega},\tilde{\calF},\tilde{\bbP};\scrX) $, and consequently convergent. So (6) holds, witnessed by $p=2$. 
\end{proof}

\begin{lemma}[Kahane-Khintchine]
	(6) $\iff$ (7). 
	\label{lem:Khintchine}
\end{lemma}
\begin{proof}  One formulation of the Kahane-Khintchine inequality -- cf.\ \cite[Proposition 2.7]{vanNeerven2010} --  states that for each $p\in [1,\infty)$ there exist universal constants $c_p,C_p>0$ such that 
	\begin{equation}
	c_p\Big\lVert \sum_{n=N}^M \tilde{\gamma}_n(-) x_n \Big\rVert_{L^p(\tilde{\Omega},\tilde{\calF}, \tilde{\bbP};\scrX)} \leq 
	\Big\lVert \sum_{n=N}^M \tilde{\gamma}_n(-) x_n \Big\rVert_{L^2(\tilde{\Omega},\tilde{\calF}, \tilde{\bbP};\scrX)}  \leq C_p \Big\lVert \sum_{n=N}^M \tilde{\gamma}_n(-) x_n \Big\rVert_{L^p(\tilde{\Omega},\tilde{\calF}, \tilde{\bbP};\scrX)}.
	\label{eq:hz2}
	\end{equation}
	(By H\"older's inequality, one of $c_p,C_p$ is equal to one, depending on whether $p\leq 2$ or $p\geq 2$.)

	Given this, the sequence $\{\Sigma_N \}_{N=0}^\infty$ is Cauchy and hence convergent in $L^2(\tilde{\Omega},\tilde{\calF}, \tilde{\bbP};\scrX)$ if and only if it is Cauchy  in $L^p(\tilde{\Omega},\tilde{\calF}, \tilde{\bbP};\scrX)$, and so we can conclude \Cref{lem:Khintchine}. 
\end{proof}

Recall the Paley-Zygmund inequality, which says that given a nonnegative random variable $Z: \tilde{\Omega}\to [0,\infty]$ with finite variance,
\begin{equation}
\tilde{\bbP} [Z > \theta \bbE Z] \geq (1-\theta)^2 \bbE [Z]^2/\bbE [Z^2].
\label{eq:PZ}
\end{equation}
for all $\theta \in [0,1]$. 
\Cref{eq:PZ} gives a quick proof of the implication (5) $\Rightarrow$ (6), (7) which can be substituted in for Fernique's theorem or the converse of Sazonov's theorem in a proof of the implication (2) $\Rightarrow$ (1). 
\begin{lemma}
	(5) $\Rightarrow$ (6), (7).
	\label{lem:4->5}
\end{lemma}
We follow the argument in \cite[Chapter 6, Corollary  6.2.9]{Hytonen2016analysis}. The following makes sense for $\scrX$ any separable Banach space. 
\begin{proof}
	Suppose that (5) holds, so that the partial series $\smash{\Sigma_N=\sum_{n=0}^N \tilde{\gamma}_n x_n}$ converges strongly in probability to some $\scrX$-valued random variable $\smash{\Sigma:\tilde{\Omega}\to \scrX}$. Fix $\delta>0$ and to be decided $\epsilon>0$. Pick $N_0\in \bbN$ sufficiently large such that 
	\begin{equation} 
	\tilde{\bbP}[ \lVert \Sigma_{N'} - \Sigma_{N''} \rVert_\scrX > \delta/2] < \epsilon 
	\end{equation} 
	for all $N',N'' \geq N_0$. Let $Z = \lVert \Sigma_{N'} - \Sigma_{N''} \rVert^p_\scrX$ for $p\in [1,\infty)$, $Z:\tilde{\Omega}\to [0,\infty)$. Clearly, this has finite variance. Our goal is to show that if we choose $\epsilon = \epsilon(\delta)$ sufficiently small, $\bbE Z$ satisfies 
	\begin{equation} 
   	\bbE Z < \delta^p,
   	\label{eq:6}
	\end{equation} 
	for such $N',N''$, 
	which by the arbitrariness of $\delta$ implies that $\{\Sigma_N\}_{N=0}^\infty \subseteq L^p(\tilde{\Omega},\tilde{\calF},\tilde{\bbP};\scrX)$ is a Cauchy hence convergent sequence in $L^p(\tilde{\Omega},\tilde{\calF},\tilde{\bbP};\scrX)$. Suppose, to the contrary, that $\bbE Z \geq \delta^p$, so that $\bbP[Z > \theta \bbE Z] = \bbP[\lVert \Sigma_{N'} - \Sigma_{N''} \rVert_{\scrX}>\delta/2] < \epsilon$ for $\theta = (\delta/2)^p/\bbE Z \in [0,1]$. By the Paley-Zygmund inequality and the Kahane-Khintchine inequality, this implies that 
	\begin{equation}
	\epsilon  > \Big( 1 - \frac{\delta^p}{2^p \bbE Z} \Big)^2 \frac{\bbE [Z]^2}{\bbE [Z^2]} \geq \Big( 1 - \frac{\delta^p}{2^p \bbE Z} \Big)^2  \frac{c_{2p}^{2p}}{C_p^{2p}},
	\end{equation}
	where $c_p,C_p$ are as in \cref{eq:hz2}.
	Solving for $\bbE Z$ yields the inequality $(\delta/2)^p  > (1- \epsilon^{1/2} C_p^p c_{2p}^{-p}) \bbE Z$.
	 So taking $\epsilon$ sufficiently small,  \cref{eq:6} does indeed hold (and contradicting our assumption to the contrary).
\end{proof}

The previous argument yields the implication (5) $\Rightarrow$ (1) via \Cref{lem:6->1}. 
We now consider the implication (5) $\Rightarrow$ (1) using instead Fernique's theorem on the exponential integrability of Gaussian measures, and we prove the last statement in \Cref{prop:Sazonov_equiv}. Fernique's theorem  states that if $\scrX$ is a separable Banach space, then a centered $\scrX$-valued Gaussian random variable is exponentially integrable for some small base (in the sense of \cref{eq:za}), and in particular has moments of all finite orders. See  \cite{Fernique}\cite[\S2.2]{Prato}\cite[Chapter 6]{Hytonen2016analysis}. 

In the following arguments, we need the following: a sequence 
\begin{equation} 
\{\Gamma_N:\operatorname{Borel}(\bbR^D)\to [0,1]\}_{N=0}^\infty
\end{equation} 
of (possibly degenerate) Gaussian measures on finite dimensional Euclidean space converges in law if and only if the means ${\bm\mu}_N = \int_{\bbR^d} x \dd \Gamma_N(x) \in \bbR^D$ and covariances 
\begin{equation} 
	(\bmsigma_N)_{i,j} = \int_{\bbR^d} x_i x_j \dd \Gamma_N(x) \in \bbR^{D\times D}
\end{equation} 
converge as $N\to\infty$, in which case the limit is a (possibly degenerate) Gaussian measure with mean $\smash{\lim_{N\to\infty} {\bmmu}_N}$ and covariance matrix $\smash{\bmsigma = \lim_{N\to\infty} \bmsigma_N}$. This can be proven directly or alternatively using L\'evy's continuity theorem \cite[\S26]{Billingsley}.  
\begin{lemma} \label{Fernique_lemma}
	If (4) holds, then the limit $\Sigma:\tilde{\Omega}\to \scrX$ is Gaussian, which by Fernique's theorem implies that there exists some constant $\alpha>0$ such that 
	\begin{equation}
	e^{- \alpha \lVert \Sigma \rVert_{\scrX}^2} \in L^1(\tilde{\Omega},\tilde{\calF},\tilde{\bbP};\scrX). 
	\label{eq:za}
	\end{equation}
	Moreover, $\lVert \Sigma \rVert_\scrX$ has moments of all orders and (1) holds. 
\end{lemma}
\begin{proof}
	Suppose that (4) holds, so that there exists a random variable $\Sigma:\tilde{\Omega}\to \scrX$ such that
	$\langle \Sigma_N, x^* \rangle_\scrX$ converges to $\langle \Sigma,x^*\rangle_\scrX$ in probability for each individual $x^*\in \scrX$ and therefore converges in distribution. Each individual 
	\begin{equation} 
	\langle \Sigma_N ,x^* \rangle_\scrX = \sum_{n=0}^N \tilde{\gamma}_n \langle x_n ,x^* \rangle_\scrX  
	\end{equation} 
	is a sum of independent centered Gaussian random variables and therefore has a centered Gaussian law. This implies that $\langle \Sigma,x^* \rangle_\scrX$ has a centered Gaussian law.  That is, $\Sigma$ is a Gaussian $\scrX$-valued random variable. 
	We deduce \cref{eq:za} via Fernique's theorem. 
	
	To see that this implies (1), first note that it implies that $\lVert \Sigma \rVert_{\scrX}$ has all moments. In particular, 
	$\bbE (\lVert \Sigma \rVert^2_{\scrX})<\infty$. By the monotone convergence theorem, 
	\begin{equation}
	\bbE (\lVert \Sigma \rVert^2_{\scrX}) = \sum_{m=1}^\infty \bbE ( |\langle \Sigma , x^*_m \rangle_\scrX|^2 ) \label{eq:43h}
	\end{equation}
	for $\{x_m^* \}_{m=1}^\infty$ any orthonormal basis of $\scrX$. Since $\langle \Sigma , x^*_m \rangle_\scrX$ is Gaussian, the convergence in law of the Gaussian $\langle \Sigma_N , x^*_m \rangle_\scrX$ to it as $N\to\infty$ implies that $\bbE ( |\langle \Sigma_N  , x^*_m \rangle_\scrX|^2) \to \bbE ( |\langle \Sigma,x^*_m \rangle_{\scrX}|^2)$ as $N\to\infty$. Computing out the former,
	\begin{equation}
	\bbE (| \langle \Sigma,x^*_m \rangle_{\scrX}|^2) = \sum_{n=0}^\infty |\langle x_n ,x_m^* \rangle_{\scrX}|^2. 
	\end{equation}
	Plugging this into \cref{eq:43h}, 
	\begin{equation}
	\bbE (\lVert \Sigma \rVert^2_{\scrX})  = \sum_{m=1}^\infty \sum_{n=0}^\infty |\langle x_n , x_m^* \rangle_\scrX|^2 = \sum_{n=0}^\infty \sum_{m=1}^\infty  |\langle x_n , x_m^* \rangle_\scrX|^2 = \sum_{n=0}^\infty \lVert x_n \rVert_{\scrX}^2 . 
	\label{eq:xxv}
	\end{equation}
	So, $\sum_{n=0}^\infty \lVert x_n \rVert^2_\scrX$ is finite. 
\end{proof}

We can strengthen \Cref{Fernique_lemma} to only require the verification of (4) on a dense subset $S$ of functionals, which we assume without loss of generality is a cone in $\scrX^*\cong \scrX$.  This yields the penultimate statement in \Cref{prop:Sazonov_equiv}. In order to accomplish this, we can construct orthonormal $x_m^* \in S$ via the Gram-Schmidt process. For the proof, it suffices to note that an $\scrX$-valued random variable $\Sigma:\tilde{\Omega}\to \scrX$ is Gaussian if  $\langle \Sigma,x^* \rangle_\scrX$ is Gaussian for all $x^*$ in some dense subset of $\scrX$. In order to conclude the result, we introduce a bit of Fourier analysis. 
Given a separable real Banach space $\scrY$ and a Borel probability measure 
\begin{equation} 
\mu : \operatorname{Borel}(\scrY)\to [0,1]
\end{equation} 
on $\scrY$ -- which we recall (like any other Borel probability measure on any Polish space) is automatically Radon -- its \emph{Fourier transform} is the function $\calF\mu : \scrY^* \to \bbC$ given by 
\begin{equation}
\calF \mu (\Lambda) = \int_{\scrY} e^{- i \Lambda(y)} \dd \mu(y). 
\label{eq:fu}
\end{equation}
We work with real Banach spaces so that the right-hand side of \cref{eq:fu} is well-defined. 
It can be shown -- for instance using the inner regularity of $\mu$ (or \cite{Bogachev} the dominated convergence theorem and the equivalence of sequential continuity and continuity in a metric space) -- that $\calF \mu$ is continuous with respect to the operator norm topology on $\scrY^*$. See \cite[Part II-Chapter II]{SchwartzRadon} for a discussion in full generality.

We will apply this definition with $\scrY = \scrX_\bbR$ the real Banach space underlying $\scrX$, which for the sake of generality we can take to be an arbitrary separable complex Banach space. We conflate $(\scrX^*)_\bbR \cong (\scrX_\bbR)^*$.
\begin{lemma}
	If $\scrX$ is a separable complex Banach space and $S\subseteq \scrX^*$ is a dense subset, then an $\scrX$-valued random variable $\Sigma:\tilde{\Omega}\to \scrX$ is Gaussian if and only if $\Lambda\circ \Sigma$ is a $\bbC$-valued Gaussian random variable for each $\Lambda \in S$. 
	\label{density_lemma}
\end{lemma}
\begin{proof} We may assume without loss of generality that $S$ is a cone in $\scrX^*$, meaning that it is closed under scalar multiplication. Note that $S_\bbR$ is dense in $\scrX^*_\bbR$. 
	Let $\Lambda \in \scrX_\bbR^*$, $\Lambda_1,\Lambda_2,\ldots \in S_\bbR$ converge to $\Lambda$ in operator norm. For each $\lambda \in \bbR$, $\lambda \Lambda_m \to \lambda \Lambda$ in operator norm. Suppose that $\Sigma:\tilde{\Omega}\to \scrX$ is a random variable satisfying the hypothesis of the lemma. By the continuity of the Fourier transform  $\calF \Sigma_* \tilde{\bbP}$ of the law of $\Sigma$ with respect to the operator norm, 
	\begin{equation}
	(\calF \Sigma_* \tilde{\bbP}) (\lambda \Lambda_m) \to 
	(\calF \Sigma_* \tilde{\bbP}) (\lambda \Lambda) 
	\label{eq:10}
	\end{equation}
	as $m\to\infty$ for each $\lambda \in \bbR$. The function on the left-hand side of \cref{eq:10} is a (possibly degenerate) Gaussian function of $\lambda$ for each $m\in \bbN^+$, and so $\lambda\mapsto \calF \Sigma_* \tilde{\bbP} (\lambda \Lambda) $ is a pointwise limit of Gaussian functions on $\bbR$. It follows that $\calF \Sigma_* \tilde{\bbP} (\lambda \Lambda)$ is a Gaussian or constant, and therefore that $\Lambda(\Sigma)$ is a (possibly degenerate) $\bbR$-valued Gaussian random variable. 
	
	Since $\Lambda \in \scrX^*_\bbR$ was arbitrary, we conclude that $\Sigma$ is a Gaussian $\scrX_\bbR$-valued random variable, which implies that it is a Gaussian $\scrX$-valued random variable.
\end{proof}

We note that the equivalence of (2), (3), (4) is a very special case of the It\^o-Nisio theorem \cite{ItoNisio}\cite{Hytonen2016analysis}: 
\begin{proof}[Proof of (2) , (4) $\Rightarrow$ (3) from It\^o-Nisio]
	The It\^o-Nisio theorem says in particular that (2), (3), and (4) are equivalent even if $\scrX$ is replaced by any separable $\bbC$-Banach space and the Gaussians are replaced by arbitrary but still independent symmetric random variables. 
\end{proof}
We now have one complete proof of \Cref{prop:Sazonov_equiv}. 

The following slight strengthening of the It\^o-Nisio theorem due to Hoffmann-J{\o}rgensen is useful (but not strictly necessary for the proof of \Cref{prop:Sazonov_equiv} given \Cref{density_lemma}). 
Given a dense set of functionals $S\subseteq \scrX^*$, where $\scrX$ is still an arbitrary separable complex Banach space, we consider the topology $\tau_S$ on the vector space $\scrX$ generated by the elements of $S$.
Clearly, $\tau_S$ is weaker than the norm topology on $\scrX$ and furnishes $\scrX$ with the structure of a locally convex Hausdorff vector space.   (So $\tau_S$ is the topology of the LCTVS we called `$\scrX_S$' in \S\ref{subsec:random}.) Moreover, the unit ball $\{x \in \scrX: \lVert x \rVert_\scrX \leq 1\}$ of $\scrX$ is $\tau_S$-closed, as we noted earlier. The following is therefore essentially a special case of \cite[Theorem 6.2]{Hoffmann1974}. 
\begin{theoremp}
	\label{Ito-Nisio}
	Suppose that $\scrX$ is a separable Banach space and that $x_0,x_1,x_2,\ldots : \smash{ \tilde{\Omega}}\to \scrX$ are $\scrX$-valued independent symmetric random variables. If there exists an $\scrX$-valued random variable $\Sigma:\tilde{\Omega}\to \scrX$ and dense subset $S\subseteq \scrX^*$ such that 
	\begin{equation} 
		\Lambda\Big[\sum_{n=0}^N x_n\Big] \to \Lambda(\Sigma)
	\end{equation} 
	in probability for each $\Lambda \in S$, then, for $\Sigma_N = \sum_{n=0}^N x_n$,  $\lim_{N\to\infty} \Sigma_N  = \Sigma$ strongly $\tilde{\bbP}$-almost surely. 
\end{theoremp}
We sketch how this follows from a standard proof of the It\^o-Nisio theorem.
\begin{proof}[Proof Sketch]
	We may assume without loss of generality that $S$ is a subspace of $\scrX^*$. 
	As above, let $\scrX_\bbR$ denote the real Banach space underlying $\scrX$, and note that the set $S_\bbR = \{\Re \Lambda : \Lambda \in S\}$ is dense in $\{\Re \Lambda :\Lambda \in \scrX^* \}=(\scrX_\bbR)^*$. Since the Fourier transform $(\scrX_\bbR)^* \to \bbC$ of the law of $\Sigma$ is necessarily continuous, it is determined by its restriction to $S_\bbR$. 
	
	The Hahn-Banach separation theorem implies that given any $x_0\in \scrX$ and nonnegative $R < \lVert x_0 \rVert_\scrX$, there exists a $\Lambda \in S_{\bbR}$ such that $\Lambda(x)>0$ for every 
	\begin{equation} 
	x\in B_R(x_0)=\{x \in \scrX: \lVert x - x_0 \rVert_{\scrX}\leq R\}.
	\end{equation} 
	The argument in \cite[\S6.4]{Hytonen2016analysis} used to prove the It\^o-Nisio theorem therefore only needs the assumption of weak convergence in probability to be verified on the elements of $S_\bbR$. 
\end{proof}

\subsection{Second and Third Proof of \Cref{prop:Sazonov_equiv}}
\label{subsec:3.st}
Recall the Kolmogorov two-series theorem \cite[Theorem 22.6]{Billingsley}, which states that given a sequence $r_0,r_1,r_2,\ldots$ of independent $\bbC$-valued random variables on some probability space with finite mean $\bbE r_n = \mu_n$ and variance $\operatorname{Var}(r_n) \in [0,\infty]$,
\begin{equation}
\sum_{n=0}^\infty \mu_n < \infty \text{ and } \sum_{n=0}^\infty \operatorname{Var}(r_n) < \infty \Rightarrow \sum_{n=0}^N r_n \text{ converges as $N\to\infty$ almost surely}. 
\end{equation}
Given the setup of \Cref{prop:Sazonov_equiv}, we will apply the two-series theorem to the $\bbC$-valued random variables $\langle x^*, \tilde{\gamma}_n(-) x_n \rangle_\scrX$ for $x^* \in \scrX$ the elements of an orthonormal basis of $\scrX$. 

e now prove, independently of the reasoning in \S\ref{subsec:3.f}:

\begin{lemma}
	(1) $\Rightarrow$ (2). 
	\label{lem:1->2}
\end{lemma}
In order to prove this, we use the following maximal inequality. 
\begin{lemma}[L\'evy's Maximal Inequality]
	\label{lem:Levy}
	Let $\scrX$ be a separable complex Banach space. 
	Let $x_0,x_1,x_2,\ldots$ be independent symmetric $\scrX$-valued random variables. Then, setting $\smash{\Sigma_N  = \sum_{n=0}^N x_n}$, 
	\begin{equation}
	\bbP[(\exists N_0 \in \{0,\ldots,N\})\lVert \Sigma_{N_0} \rVert_\scrX >  R] \leq 2 \bbP[\lVert \Sigma_N \rVert_{\scrX} > R]
	\end{equation}
	for all $N\in \bbN$ and real $R>0$. In particular, (by countable additivity,) $\bbP[(\exists N_0 \in \bbN) \lVert \Sigma_{N_0} \rVert_\scrX >  R ] \leq 2 \liminf_{N\to \infty} \bbP[\lVert \Sigma_N \rVert_{\scrX} > R]$. 
\end{lemma}

See \cite[Proposition 6.1.12]{Hytonen2016analysis} for a proof. We will use the fact observed above that a bounded sequence in a Banach space converges weakly if and only if it converges $S$-weakly for a dense set of functionals $S\subseteq \scrX^*$. 

\begin{proof}[Proof of \Cref{lem:1->2}]
	Set $r_n=\langle x^*, \tilde{\gamma}_n(-) x_n \rangle_\scrX$ for an arbitrary unit  vector $x^* \in \scrX$. Clearly, $\bbE r_n = 0$, and by Cauchy-Schwarz we have an upper bound
	\begin{equation}
	\operatorname{Var}(r_n)  = \bbE [|\langle x^*, \tilde{\gamma}_n(-) x_n \rangle_\scrX|^2] \leq \lVert x_n \rVert_\scrX^2  \bbE [\tilde{\gamma}_n^2] = \lVert x_n \rVert_{\scrX}^2. 
	\end{equation}
	Assumption (1) implies the remaining hypothesis of the two-series theorem, and therefore that
	\begin{equation}
	\sum_{n=0}^N \langle x^* , \tilde{\gamma}_n(-) x_n \rangle_\scrX = \Big\langle x^* ,
	\sum_{n=0}^N \tilde{\gamma}_n(-) x_n \Big\rangle_\scrX 
	\label{eq:z0}
	\end{equation}
	converges as $N\to\infty$ almost surely. Let $\Sigma_N = \Sigma_{\infty,N} = \sum_{n=0}^N \tilde{\gamma}_n x_n$, let $\{x^{*}_m\}_{m=1}^\infty$ denote an orthonormal basis of $\scrX$, and let $S \subseteq \scrX$ denote the set consisting of all linear combinations of finitely many of the $x^*_m$'s with coefficients in $\bbQ(i)$. $S$ is dense in $\scrX$ and countable. 
	Almost surely, the preceding convergence holds for all $x^* \in S$ simultaneously. 
	In particular, there exist $\bbC$-valued random variables $a_1,a_2 ,a_3,\ldots : \tilde{\Omega}\to \bbC$ such that 
	\begin{equation}
	\sum_{n=0}^\infty  \langle x^{*}_m ,\tilde{\gamma}_n(-)
	 x_n \rangle_\scrX = a_m \text{ for all $m$}
	\label{eq:am}
	\end{equation}
	almost surely. Consider the random variable $\Sigma_{M,\infty}:\tilde{\Omega}\to \scrX$ defined by
	\begin{equation}
	\Sigma_{M,\infty} = \sum_{m=1}^M a_m x^{*}_m .
	\end{equation}
	For fixed $\tilde{\omega}\in \tilde{\Omega}$, $\lim_{M\to\infty} \Sigma_{M,\infty}(\tilde{\omega})$ exists as a well-defined element of $\scrX$ if and only if $\sum_{m=1}^\infty |a_m(\tilde{\omega})|^2$ is finite. 
	
	The infinite series $ \sum_{m=1}^\infty |a_m|^2$ is a well-defined $[0,\infty]$-valued random variable, and so will be almost surely finite if  it has finite expectation value. In this case, we conclude that 
	\begin{equation} \lim_{M\to\infty} \Sigma_{M,\infty} = \Sigma_{\infty,\infty}  \in \scrX
	\label{eq:kkk}
	\end{equation}
	exists and satisfies $\sum_{m=1}^\infty |a_m|^2 = \lVert \Sigma_{\infty,\infty} \rVert_{\scrX}^2$ almost surely. 
	We check that $\bbE  \sum_{m=1}^\infty |a_m|^2 < \infty$. The monotone convergence theorem implies that 
	\begin{equation}
	\bbE \sum_{m=1}^\infty |a_m|^2 = \lim_{M\to \infty} \int_{\tilde{\Omega}} \sum_{m=1}^M |a_m|^2 \dd \tilde{\bbP} =  \sum_{m=1}^\infty \int_{\tilde{\Omega}}  |a_m|^2 \dd\tilde{\bbP} =\sum_{m=1}^\infty \bbE [|a_m|^2]. 
	\label{eq:l9}
	\end{equation}
	Applying Fatou's lemma,
	\begin{align}
	\sum_{m=1}^\infty \bbE [|a_m|^2] &\leq \sum_{m=1}^\infty \liminf_{N\to\infty} \int_{\tilde{\Omega}}  \Big|\Big\langle x^*_m ,
	\sum_{n=0}^N \tilde{\gamma}_n(\tilde{\omega}) x_n\Big\rangle_\scrX \Big|^2 \dd \tilde{\bbP} (\tilde{\omega}) \\ &=  \sum_{m=1}^\infty \liminf_{N\to\infty} \int_{\tilde{\Omega}}  \Big|
	\sum_{n=0}^N \tilde{\gamma}_n(\tilde{\omega}) \langle x^*_m ,  x_n\rangle_\scrX \Big|^2 \dd \tilde{\bbP} (\tilde{\omega}). 
	\label{eq:zz}
	\end{align}
	Because $\tilde{\gamma}_0,\tilde{\gamma}_1,\tilde{\gamma}_2,\ldots$ are independent and have mean zero, we can calculate the right-hand side of \cref{eq:zz}. The result is the bound 
	\begin{equation}
	\bbE \sum_{m=1}^\infty |a_m|^2 \leq  \sum_{m=1}^\infty \liminf_{N\to\infty} 
	\sum_{n=0}^N | \langle x^*_m ,  x_n\rangle_\scrX|^2 = \sum_{n=0}^\infty \sum_{m=1}^\infty |\langle x^*_m , x_n \rangle_\scrX|^2  = \sum_{n=0}^\infty \lVert x_n \rVert_\scrX^2. 
	\label{eq:39}
	\end{equation}
	By our assumption (1), the right-hand side of \cref{eq:39} is finite. We conclude that $\Sigma=\Sigma_{\infty,\infty}$, \cref{eq:kkk} is a well-defined element of $\scrX$. 
	
	We already know that $\Sigma_{\infty,N}(\tilde{\omega})$ converges to $\Sigma(\tilde{\omega})$ $S$-weakly for $\tilde{\bbP}$-almost all $\tilde{\omega}\in \tilde{\Omega}$, and since $S$ is dense in $\scrX$, in order to prove that 
	\begin{equation} 
		\Sigma_N(\tilde{\omega})=\Sigma_{\infty,N}(\tilde{\omega})\to \Sigma(\tilde{\omega})
	\end{equation} 
	weakly for $\tilde{\bbP}$-almost all $\tilde{\omega}\in \tilde{\Omega}$, it suffices to prove that $\{\Sigma_{N}(\tilde{\omega})\}_{N=0}^\infty$ is bounded in $\scrX$ for $\tilde{\bbP}$-almost all $\tilde{\omega}\in \tilde{\Omega}$.  In other words, we want to show that 
	\begin{equation}
	\tilde{\bbP}[(\forall R>0) (\exists N \in \bbN) \lVert \Sigma_N \rVert_\scrX > R] = 0.
	\end{equation}
	Since $\{\tilde{\omega}\in \tilde{\Omega} : (\forall R>0) (\exists N \in \bbN) \lVert \Sigma_N \rVert_\scrX > R\} = \cap_{R>0} \{\tilde{\omega}\in \tilde{\Omega} : (\exists N \in \bbN) \lVert \Sigma_N \rVert_\scrX > R\}$ and 
	\begin{equation} 
		\{\tilde{\omega}\in \tilde{\Omega}:(\exists N \in \bbN) \lVert \Sigma_N \rVert_\scrX > R\} = \cup_{N>0} \{\tilde{\omega}\in \tilde{\Omega}:(\exists N_0 \in \{0,\ldots,N\}) \lVert \Sigma_{N_0} \rVert_\scrX > R\}
	\end{equation} 	 
	 (so that by L\'evy's maximal inequality $\bbP[(\exists N \in \bbN) \lVert \Sigma_N \rVert_\scrX > R] \leq 2 \liminf_{N\to\infty} \tilde{\bbP}[\lVert \Sigma_N \rVert_\scrX>R]$ for all $R>0$), it suffices to prove that 
	\begin{equation}
	 \liminf_{N\to\infty} \tilde{\bbP}[\lVert \Sigma_N \rVert_{\scrX} > R] \to 0. 
	 \label{eq:j}
	\end{equation}
	as $R\to \infty$. By Markov's inequality, $ \tilde{\bbP}[\lVert \Sigma_N \rVert_{\scrX} > R] =  \tilde{\bbP}[\lVert \Sigma_N \rVert_{\scrX}^2 > R^2] \leq (1/R^2) \bbE [ \lVert \Sigma_N \rVert_{\scrX}^2]$ for $R>0$.   But $\sup_N \bbE [\lVert \Sigma_N \rVert_{\scrX}^2] = \sum_{n=0}^\infty \lVert x_n \rVert_{\scrX}^2 < \infty$ by assumption,  and so \cref{eq:j} holds. 
\end{proof}

We can upgrade the preceding proof of (1) $\Rightarrow$ (2) to a proof of (1) $\Rightarrow (3)$, i.e.\ to conclude almost sure strong convergence, bypassing the It\^o-Nisio theorem. This argument uses a trick of Schwartz, 
based on the following deterministic functional analytic lemma:

\begin{lemma}
	\label{lem:fa}
	Given any Hilbert-Schmidt map $L:\scrY\to \scrX$ of separable Hilbert spaces, there exists a separable Hilbert space $\scrX_0$, compact $i:\scrX_0\to \scrX$, and Hilbert-Schmidt $L_0:\scrY\to \scrX_0$ such that $L = i\circ L_0$. 
\end{lemma}
See \cite[Part II-Chp.\ III, Proposition 6]{SchwartzRadon} for a proof (from a theorem of duBois-Reymond).

\begin{proof}[Proof of (1) $\Rightarrow$ (3) from (1) $\Rightarrow$ (2) and \Cref{lem:fa}]
	Suppose that (1) holds. 
	
	Then we can define a HS map $L:\ell^2(\bbN)\to \scrX$ such that $L{\bmdelta}_n= x_n$. 
Applying \Cref{lem:fa} with 	$\scrY= \ell^2(\bbN)$, we get $\scrX_0$, compact $i:\scrX_0\to \scrX$, and HS $L_0:\ell^2(\bbN)\to \scrX_0$ with $L=i\circ L_0$. Because $L_0$ is HS, 
\begin{equation}
\sum_{n=0}^\infty \lVert L_0 \bmdelta_n \rVert_{\scrX_0}^2 < \infty. 
\end{equation}
We can therefore apply the implication (1) $\Rightarrow$ (2) with $\scrX_0$ in place of $\scrX$ and $L_0 \bmdelta_n$ in place of $x_n$. This tells us that 
\begin{equation}
\sum_{n=0}^N \tilde{\gamma}_n(\tilde{\omega}) L_0 \bmdelta_n \text{ converges weakly in $\scrX_0$ as $N\to\infty$}
\end{equation}
for $\tilde{\bbP}$-almost all $\tilde{\omega}\in\tilde{\Omega}$. 
Compact maps between Banach spaces map weakly convergent sequences to norm convergent sequences \cite[Theorem 3.1.9]{Simon2015operator}, so we conclude that 
\begin{equation}
	i\Big(\sum_{n=0}^N \tilde{\gamma}_n(\tilde{\omega}) L_0 \bmdelta_n \Big) = \sum_{n=0}^N \tilde{\gamma}_n(\tilde{\omega}) x_n \text{ converges strongly in $\scrX$ as $N\to\infty$}
\end{equation}
for those same $\tilde{\omega}$. 
Therefore (3) holds. 
\end{proof}

We now have a second complete proof of \Cref{prop:Sazonov_equiv}. 
We finally turn to using Sazonov's theorem (really its converse), which we take to be a classification of the covariance operators of Gaussian measures on a separable Hilbert space. (It is not difficult to translate this perspective into one regarding cylinder-set measures, and the interested reader is referred to \cite{SchwartzRadon}. Another account in terms of $\gamma$-radonifying operators is given in \cite{vanNeerven2010}.)

If $\mu$ is a Borel probability measure on a real separable Banach space $\scrY$ and has finite second moment, then its \emph{covariance matrix} is the sesquilinear form $C:\scrY^*\times \scrY^*\to \bbC$ given by 
\begin{equation}
C(\Lambda_1,\Lambda_2) = \int_{\scrY} \Lambda_1(y)^* \Lambda_2(y)\dd \mu(y). 
\end{equation}
(This is well-defined by the assumption on the second moment and Cauchy-Schwarz.)
The map $C$ is jointly continuous due to the inner regularity of $\mu$. 
When $\scrY$ is a real Hilbert space, we identify $\scrY^*$  with $\scrY$ in the definition above. When $\scrY$ is the real Hilbert space underlying our complex separable Hilbert space $\scrX$, $C$ defines via joint continuity or via the Hellinger-Toeplitz theorem \cite[\S3.5]{RS} a self-adjoint bounded $\bbC$-linear map $Q:\scrX\to \scrX$ satisfying $C(x,y) = \langle Qx , y \rangle_\scrX$. This is the \textit{covariance operator}. Note that $Q$ is positive semidefinite and self-adjoint.

One formulation of the converse of Sazonov's theorem says that a bounded positive semidefinite self-adjoint operator on $\scrX$ is a covariance matrix of some Gaussian measure on $\scrX$ if and only if it is of trace class.

\begin{proof}[Proof of (4) $\Rightarrow$ (1) using Sazonov's theorem]
	As above, let $\Sigma : \tilde{\Omega}\to \scrX$ be the random variable to which the partial series $\smash{\Sigma_N = \sum_{n=0}^N \tilde{\gamma}_n x_n}$, \cref{eq:sn}, converge weakly in $\scrX$ in probability. Also as above, $\Sigma$ is Gaussian. 
	The covariance operator of the law of $\Sigma$ is given, via (4), by 
	\begin{equation}
	C(x,y) = \bbE [ \langle x, \Sigma \rangle_\scrX \langle \Sigma,y \rangle_\scrX] = \lim_{N\to\infty}  \bbE [ \langle x, \Sigma_N \rangle_\scrX \langle  \Sigma_N,y \rangle_\scrX] 
	\end{equation}
	for $x,y \in \scrX$. We can directly compute 
	\begin{equation}
	 \bbE [ \langle x, \Sigma_N \rangle_\scrX \langle  \Sigma_N,y \rangle_\scrX] = \sum_{n=0}^N \langle x , x_n \rangle_\scrX \langle x_n ,y \rangle_\scrX,
	 \label{eq:misc_010}
	\end{equation}
	so $C(x,y) = \sum_{n=0}^\infty \langle x,x_n \rangle_\scrX \langle x_n,y\rangle_\scrX$. Consequently, 
	\begin{equation} 
	 \langle x,Q y\rangle_\scrX = C(x,y) = \lim_{N\to\infty}\Big\langle x,  \sum_{n=0}^N x_n \langle x_n,y \rangle_\scrX \Big\rangle_\scrX.
	\end{equation} 
	Since $x$ was arbitrary, we conclude that $\smash{\sum_{n=0}^N \langle x_n,y \rangle_\scrX x_n \to Qy}$ weakly as $N\to\infty$.
	
	Fixing an arbitrary orthonormal basis $\{x_m^*\}_{m=1}^\infty$ of $\scrX$, the trace of $Q$ is  given by 
	\begin{align} 
	\begin{split} 
	\operatorname{Tr}Q = \sum_{m=1}^\infty \langle x_m^* ,Q x_m^* \rangle_\scrX  &=\sum_{m=1}^\infty \lim_{N\to\infty}\Big\langle x_m^* , \sum_{n=0}^N x_n \langle x_n,x_m^* \rangle_\scrX  \Big\rangle_\scrX    \\
	&=\sum_{m=1}^\infty \sum_{n=0}^\infty |\langle x_m^* ,x_n \rangle_\scrX|^2 \\ &= \sum_{n=0}^\infty \sum_{m=1}^\infty  |\langle x_m^* ,x_n \rangle_\scrX|^2= \sum_{n=0}^\infty \lVert x_n \rVert_{\scrX}^2. 
	\end{split}
	\label{eq:Sazonov_final}
	\end{align} 
	By Sazonov's theorem, $Q$ is of trace class, and so we conclude (1). 
\end{proof}

This completes our discussion of \Cref{prop:Sazonov_equiv}.

\section{Proof of main result}
\label{sec:main_theorem}

Recall the definition of the subspace $\Psi^s_{\mathrm{cl}}(M) \subseteq \Psi^s(M)$ of $s$th ($s\in \bbR$) order Kohn-Nirenberg classical pseudodifferential operators on $M$. This is not strictly necessary for our argument, but it does simplify matters somewhat. 
These are the pseudodifferential operators on $M$ whose Schwartz kernels are smooth except (possibly) at the diagonal, where they are in local coordinates given by the Schwartz kernels of elements of 
\begin{equation} 
\Psi^s_{\infty,\mathrm{cl}}(\bbR^d) = \{\text{$s$th order classical Kohn-Nirenberg $\Psi$DOs on $\bbR^d$}\}.
\end{equation}  
See \cite{Vasy18}\cite{HintzNotes}. 
Consider the zeroth order component of the principal symbol short exact sequence. If we restrict attention to classical zeroth order operators, then we are left with a short exact sequence 
\begin{equation}
0 \to \Psi^{-1}_{\mathrm{cl}}(M)  \hookrightarrow \Psi^0_{\mathrm{cl}}(M) \overset{\sigma^0}{\longrightarrow} S^{[0]}_{\mathrm{cl}}(T^* M) \to 0. 
\label{eq:8g3f}
\end{equation}
The quotient $S^{[0]}_{\mathrm{cl}}(T^* M) = S^0_{\mathrm{cl}}(T^* M)/ S^{-1}_{\mathrm{cl}}(T^* M)$ is naturally identifiable with the space $C^\infty(\bbS^*M)$ of smooth complex-valued functions on the cosphere bundle over $M$. 
We will choose an arbitrary (and noncanonical) right-inverse $\operatorname{Op}: C^\infty(\bbS^*M) \hookrightarrow \Psi^0_{\mathrm{cl}}(M)$ of $\sigma^0$ in \cref{eq:8g3f}, so that $(\sigma^0 \circ \operatorname{Op})(\chi) = \chi$ for all $\chi \in C^\infty(\bbS^*M)$ and with the additional property that if $\chi$ vanishes identically on a given open subset $U\subseteq \bbS^* M$ then 
\begin{equation} 
U \cap \operatorname{WF}'(\operatorname{Op}(\chi)) = \varnothing. 
\label{eq:i9}
\end{equation} 
(Clearly, at least one such ``quantization'' exists.) The characteristic and elliptic sets of $\operatorname{Op}(\chi)$ are then given by 
\begin{equation}
\operatorname{Char}^s(\operatorname{Op}(\chi)) = \chi^{-1}(\{0\}) \quad \text{ and } \quad  \operatorname{Ell}^s(\operatorname{Op}(\chi)) = \chi^{-1}(\bbC\backslash \{0\})
\end{equation}
respectively. Also, by \cref{eq:i9}, the essential support of $\operatorname{Op}(\chi)$ is given by $\operatorname{WF}'(\operatorname{Op}(\chi)) = \operatorname{supp} \chi$. 

We will consider, for $\ell_1,\ell_2\in \bbR$, the composition
\begin{equation} 
\Psi^s_{\mathrm{cl}}(M)\ni  (1+\triangle_g)^{\ell_1}\circ  \operatorname{Op}(\chi)\circ (1+\triangle_g)^{\ell_2} : \scrD'(M)\to \scrD'(M),
\label{eq:k90}
\end{equation} 
which restricts to a bounded linear operator on $\scrH^\sigma(M,g)$ for $\ell_1+\ell_2\leq 0$ and arbitrary $\sigma\in \bbR$.  
The $\scrH^s$-wavefront set of $u\in \scrD'(M)$, as defined by \cref{eq:wf}, is given in terms of the elements of $\Psi^0_{\mathrm{cl}}(M)$ by 
\begin{align}
\operatorname{WF}^s(u) &= \bigcap \{ \operatorname{Char}^0(\operatorname{Op}(\chi)) : \chi \in C^\infty(\bbS^* M), \operatorname{Op}(\chi) u \in \scrH^s(M) \} 
\label{eq:wavefront_definition_one}
\\
&= \bigcap \{ \operatorname{Char}^0(\operatorname{Op}(\chi)) : \chi \in C^\infty(\bbS^* M), (1+\triangle_g)^{s/2}\operatorname{Op}(\chi) u \in \scrL^2(M) \} \label{eq:wavefront_definition_two}\\
&= \bigcap \{ \operatorname{Char}^0(\operatorname{Op}(\chi)) : \chi \in C^\infty(\bbS^* M), \operatorname{Op}(\chi) (1+\triangle_g)^{s/2} u \in \scrL^2(M) \}.
\label{eq:wavefront_definition_three}
\end{align}
(The equivalence of \cref{eq:wavefront_definition_one} with \cref{eq:wf} is a consequence of the abundance of elliptic classical symbols and elliptic regularity. In particular, given any point in the cosphere bundle and a neighborhood of it, there is a 0th order classical symbol which is elliptic at the point and essentially supported within the neighborhood.
The equality of the two sets on the right-hand sides of \cref{eq:wavefront_definition_two} and \cref{eq:wavefront_definition_three} is a consequence of the microlocality and ellipticity of $(1+\triangle_g)^{s/2}$.) Substituting in 
\begin{equation} 
\operatorname{Char}^0(\operatorname{Op}(\chi)) = \chi^{-1}(\{0\})
\end{equation} 
to \cref{eq:wavefront_definition_three}, 
\begin{equation}
\operatorname{WF}^s(u) = \bbS^*M \backslash \bigcup\{ \chi^{-1}(\bbC\backslash \{0\}) : \chi \in C^\infty(\bbS^*M) , \operatorname{Op}(\chi) (1+\triangle_g)^{s/2} u \in \scrL^2(M)\}.
\label{eq:lk}
\end{equation}
We use the characterization \cref{eq:lk} with $u = \Phi(\omega)$, $\omega \in \Omega$ fixed. 

In order to apply \Cref{Sazonov}, we need to determine when $\operatorname{Op}(\chi) (1+\triangle_g)^{s/2}$ is a Hilbert-Schmidt operator on the various Sobolev spaces. Since the compositions of Hilbert-Schmidt operators with bounded operators are Hilbert-Schmidt, 
\begin{multline}
\operatorname{Op}(\chi) (1+\triangle_g)^{s/2} : \scrH^\sigma(M,g)\to \scrH^\sigma(M,g)\text{ is HS } \\ \iff (1+\triangle_g)^{\sigma/2} \operatorname{Op}(\chi) (1+\triangle_g)^{(s-\sigma)/2} : \scrL^2(M,g)\to \scrL^2(M,g)\text{ is HS.}
\end{multline}
It therefore suffices, if we consider the two parameter family \cref{eq:k90} of operators  $(1+\triangle_g)^{\ell_1}\operatorname{Op}(\chi) (1+\triangle_g)^{\ell_2}$ for $\ell_1,\ell_2 \in \bbR$, to restrict attention to $\scrL^2(M,g)$ among all the $L^2$-based Sobolev spaces. 

We make use of the following general characterization of Hilbert-Schmidt operators on $\scrL^2(M,g)$: a bounded linear operator $L:\scrL^2(M,g)\to \scrL^2(M,g)$ is Hilbert-Schmidt if and only if its Schwartz kernel $K_L \in \scrD'(M\times M)$ satisfies 
\begin{equation}
	K_L \in \scrL^2(M\times M,g\times g),
\end{equation}
where $g\times g$ is the product metric on $M\times M$. 
Indeed, the Hilbert-Schmidt norm $\lVert L \rVert_{\mathrm{HS}} \in [0,\infty]$ is given by 
\begin{equation}
\lVert L \rVert_{\mathrm{HS}}^2 = \sum_{n=0}^\infty \lVert L \phi_n \rVert_{\scrL^2(M,g)}^2 = \sum_{n,m=0}^\infty |\langle \phi_m ,L \phi_n \rangle_{\scrL^2(M,g)}|^2 = \lVert K_L \rVert_{\scrL^2(M\times M,g\times g)}^2.
\end{equation}
A well-known corollary of this elementary observation and the Parseval-Plancherel theorem (applied locally) is a characterization of the somewhere elliptic Kohn-Nirenberg $\Psi$DOs  which act as Hilbert-Schmidt operators on $\scrL^2(M,g)$. 
\begin{proposition} 
\label{HS_criterion} 	
If $s<-d/2$, any Kohn-Nirenberg pseudodifferential operator $L\in \Psi^s(M)$ restricts to a Hilbert-Schmidt operator on $\scrL^2(M,g)$. 
Conversely, if  $\operatorname{Ell}^s(L) \subseteq \bbS^*M$ is nonempty, then  $s<-d/2$ is necessary. 
\end{proposition}
Since the proof below only uses the leading order behavior of symbols, this proposition applies to more general pseudodifferential operators than the Kohn-Nirenberg $\Psi$DOs we consider here \cite[\S 27]{Shubin}. 
\begin{proof}
	We first consider $L\in \Psi_{\mathrm{c}}^s(V)\subseteq  \Psi_\infty^s(\bbR^d)$ for open $V\subseteq \bbR^d$ --- here $\Psi^s_{\mathrm{c}}(V)$ consists of the $s$th order Kohn-Nirenberg $\Psi$DOs whose Schwartz kernels are compactly supported in $V\times V$. We can therefore write the Schwartz kernel $K_L \in \scrS'(\bbR^d\times \bbR^d)$ of such an operator $L$ as an oscillatory integral 
	\begin{equation}
	K_L(x,y) = \frac{1}{(2\pi)^{d/2}} \int_{\bbR^d} e^{-i(x-y)\cdot \xi} a(x,\xi)  \dd^d \xi  
	\end{equation}
	for some $a(x,\xi)\in S^s(\bbR^d_x\times \bbR^d_\xi)$ which for some $L$-dependent $R$ vanishes if $|x|\geq R$. The Parseval-Plancherel theorem and Fubini's theorem together imply that $
	\lVert K_L \rVert_{\scrL^2(\bbR^{2d})} = \lVert a \rVert_{\scrL^2(\bbR^{2d})}$. 
	If $s<-d/2$, then $a$ satisfies \cref{eq:j9} with $-d/2-\varepsilon$ in place of $s$ for all sufficiently small $\varepsilon>0$. In particular
	\begin{equation} 
	C\overset{\mathrm{def}}=\sup_{x ,\xi \in \bbR^d} |\langle \xi \rangle^{+d/2+\varepsilon} a(x,\xi)| < \infty. 
	\end{equation} 
	The $L^2$-norm of $a$ is then bounded above by 
	\begin{equation}
	\lVert a \rVert_{\scrL^2(\bbR^{2d})} \leq  C\Big( R^d \operatorname{Vol}(\bbB^d) \int_{\bbR^d} \langle \xi \rangle^{-d-2\varepsilon} \dd^d \xi \Big)^{1/2} < \infty. 
	\end{equation}
	Consequently, $K_L \in \scrL^2(\bbR^{2d})$. 
	Conversely, if $a$ satisfies an estimate $|a(x,\xi)| \geq c\langle \xi \rangle^{-d/2}$ for some $c>0$ and all $x\in U \subset \bbR^d$ ($U$ open and nonempty, and $\xi$ sufficiently large and in some cone), $\lVert a \rVert_{\scrL^2(\bbR^{2d})} = \infty$. So, for somewhere elliptic $a$, the constraint $s<-d/2$ is necessary. 
	
	Given a finite open cover $\calU$, $M = \cup_{U\in \calU} U$, of $M$ by coordinate charts $x_U^\bullet:U\to \bbR^d$, we can write 
	\begin{equation}
	K_L(x,y) = E(x,y) + \sum_{U\in \calU} x_U^{\bullet *} K_L^{(U)} (x_U^{\bullet-1})^*
	\label{eq:local_decomposition} 
	\end{equation}
	for some $E\in \scrC^\infty(M\times M)$ and Schwartz kernels $K_L^{(U)}$ of elements of $\Psi_{\mathrm{c}}^s(x_U^{\bullet*}(U))$ \cite[Theorem 5.25]{HintzNotes}. If $s<-d/2$, each term in \cref{eq:local_decomposition} is in $\scrL^2(M\times M)$, so we conclude that $K_L \in \scrL^2(M\times M)$ as well. Conversely, if $\smash{K_L^{(U)}}$ is not in $\scrL^2(\bbR^{2d})$, then neither will $K_L$ be.
\end{proof}

\begin{theorem}
	\label{thm:main}
	If the standard deviations of our Gaussian distributions $\gamma_n \sim \operatorname{N}(\mu_n,\sigma_n)$ satisfy the polynomial asymptotics 
	\begin{equation} 
	0 < \liminf_{n\to \infty} (1+n)^\varsigma \sigma_n \leq \sup_{n\in \bbN} (1+n)^\varsigma \sigma_n < \infty, 
	\label{eq:zl}
	\end{equation} 
	and in addition $\{\mu_n\}_{n=0}^\infty \in h^{\varsigma-1/2}(\bbN)$, 
	then the Sobolev wavefront set of the random distribution $\Phi:\Omega\to \scrD'(M)$, \cref{eq:Phi}, is given by 
	\[
	\operatorname{WF}^s(\Phi(\omega)) = 
	\begin{cases}
	\varnothing & (s<d(\varsigma-1/2)) \\
	\bbS^*M & (s\geq d(\varsigma-1/2)) 
	\end{cases}
	\]
	for $\bbP$-almost all $\omega \in \Omega$. 
\end{theorem}
Compare with the global result \Cref{global_regularity_sharp}.
\begin{proof}
	We already verified in \Cref{global_regularity} that almost surely $\Phi \in \scrH^{d(\varsigma-1/2)-\varepsilon}(M)$ for any $\varepsilon>0$, which immediately implies that for any $s<d(\varsigma-1/2)$, 
	\begin{equation}
	\operatorname{WF}^s(\Phi(\omega)) = \varnothing \text{ for $\bbP$-almost all $\omega\in\Omega$.}
	\end{equation}
	It remains to prove that for $s\geq d(\varsigma-1/2)$, $\operatorname{WF}^s(\Phi) = \bbS^*M$ almost surely.

	We first reduce to the case $\mu_0,\mu_1,\mu_2,\ldots = 0$ --- letting $\tilde{\gamma}_n$ denote the centered Gaussian $\tilde{\gamma}_n= (\gamma_n - \mu_n)$,  we can write 
	\begin{equation}
	\Phi(\omega) = \sum_{n=0}^\infty \mu_n \phi_n + \sum_{n=0}^\infty \tilde{\gamma}_n \phi_n , 
	\end{equation}
	and by \Cref{sequence_Sobolev_correspondence} the first sum is in $\scrH^{d(\varsigma-1/2)}(M)$. Letting $\tilde{\Phi}(\omega)$ denote the second sum, this implies that 
	\begin{equation}
	\operatorname{WF}^s(\Phi(\omega)) = \operatorname{WF}^s(\tilde{\Phi}(\omega)) \text{ for }s \leq d(\varsigma-1/2).
	\end{equation}
	Since $\operatorname{WF}^s(\Phi(\omega)) \supseteq \operatorname{WF}^{d(\varsigma-1/2)}(\Phi(\omega))$ for all $s\geq d(\varsigma-1/2)$ (as follows immediately from the definition), if $\operatorname{WF}^{d(\varsigma-1/2)}(u) = \bbS^*M$ almost surely then for any $s\geq d(\varsigma-1/2)$, 
	\begin{equation}
	\operatorname{WF}^s(\Phi(\omega)) = \bbS^*M \text{ for $\bbP$-almost all $\omega\in\Omega$}
	\end{equation}
	as well.
	It therefore suffices to consider the case $\mu_0,\mu_1,\mu_2,\ldots = 0$, as claimed. 
	
	Given our hypothesis \cref{eq:zl}, $\sigma_n>0$ for sufficiently large $n\in \bbN$. We may assume without loss of generality that  $\sigma_n>0$ for all $n\in \bbN$, so that for $\bmsigma = \{\sigma_n\}_{n=0}^\infty$, $h_{\bmsigma}(\bbN)$ is well-defined. (The general case is trivially reduced to this one.)
	
	From the discussion above, we can write 
	\begin{equation}
	\operatorname{WF}^s(\Phi(\omega)) = \bbS^*M \iff (1+\triangle_g)^{s/2}\operatorname{Op}(\chi) \Phi(\omega) \notin \scrL^2(M) \text{ for all }0\neq \chi \in C^\infty(\bbS^*M). 
	\label{eq:condition} 
	\end{equation} 
	More generally, if $\chi$ is supported in some open neighborhood $U$ of $\bbS^*M$, then $(1+\triangle_g)^{s/2}\operatorname{Op}(\chi) \Phi(\omega) \notin \scrL^2(M) \Rightarrow \operatorname{WF}^s(\Phi(\omega)) \cap U \neq \varnothing$. 
	We can choose a countable collection $\{\chi_m\}_{m\in \bbN} \subset C^\infty(\bbS^*M)$ of nonzero smooth functions such that given any nonempty open neighborhood $U\subseteq \bbS^*M$ there is some $m\in \bbN$ such that $\chi_m$ is supported within $U$. If the right-hand side of \cref{eq:condition} holds for all $\chi = \chi_m$, then we conclude that the left-hand side is dense in $\bbS^*M$. Since -- by construction -- $\operatorname{WF}^s(\Phi(\omega))$ is an intersection of closed sets and therefore closed, we conclude that $\operatorname{WF}^s(\Phi(\omega)) = \bbS^*M$. 
	
	The upshot is that 
		\begin{equation}
	\operatorname{WF}^s(\Phi(\omega)) = \bbS^*M \iff (1+\triangle_g)^{s/2}\operatorname{Op}(\chi_m) \Phi(\omega) \notin \scrL^2(M) \text{ for all }m \in \bbN. 
	\label{eq:condition_improved} 
	\end{equation} 
	If we can prove that 
	\begin{itemize}
		\item for each fixed $m$,  $(1+\triangle_g)^{s/2}\operatorname{Op}(\chi_m) \Phi(\omega) \notin \scrL^2(M)$ for $\bbP$-almost all $\omega\in \Omega$,
	\end{itemize}
	 then we can conclude (by the countable additivity of $\bbP$) that 
	 \begin{itemize}
	 	\item 
	 	for $\bbP$-almost all $\omega \in \Omega$, $(1+\triangle_g)^{s/2}\operatorname{Op}(\chi_m) \Phi(\omega) \notin \scrL^2(M)$ for all $m \in \bbN$. 
	 \end{itemize}
 By \cref{eq:condition_improved}, this implies  that $\operatorname{WF}^s(\Phi(\omega)) = \bbS^*M$ for $\bbP$-almost all $\omega \in \Omega$, hence the result.  
 
 By \Cref{Sazonov} and the Kolmogorov 0-1 law (which applies because $\Psi(X)C^\infty(X)=C^\infty(X)$), the first of these statements is equivalent to the failure of the composition 
	\begin{equation} 
	(1+\triangle_g)^{s/2}\operatorname{Op}(\chi_m)\circ 	\Sigma :h_{\bmsigma}(\bbN) \to \scrD'(M) 
	\label{eq:j923f}
	\end{equation} 
	to be a Hilbert-Schmidt map into $\scrL^2(M,g)$. Given \cref{eq:zl},   $h_{\bmsigma}(\bbN)=h^{\varsigma}(\bbN)$ at the level of sets, and clearly \cref{eq:j923f} fails to  be a Hilbert-Schmidt map into $\scrL^2(M)$ if and only if 
		\begin{equation} 
	(1+\triangle_g)^{s/2}\operatorname{Op}(\chi_m)\circ 	\Sigma :h^{\varsigma}(\bbN) \to \scrD'(M) 
	\label{eq:j923g}
	\end{equation} 
	fails to be a Hilbert-Schmidt map into $\scrL^2(M,g)$. 
	But recall  that $\Sigma|_{h^{\varsigma}(\bbN)}:h^{\varsigma}(\bbN)\to \scrH^{d\varsigma}(M,g)$ is an equivalence of Banach spaces, so -- since the Hilbert-Schmidt operators constitute an operator ideal (see \cite[Theorem 3.8.2]{Simon2015operator}) -- \cref{eq:j923g} fails to be a Hilbert-Schmidt map into the space $\scrL^2(M,g)$ if and only if 
	\begin{equation} (1+\triangle_g)^{s/2} \operatorname{Op}(\chi_m) : \scrH^{d\varsigma}(M,g) \to \scrD'(M)
	\label{eq:j923h}
	\end{equation} 
	fails to be a Hilbert-Schmidt map into $\scrL^2(M,g)$. 
	And since $(1+\triangle_g)^{d\varsigma/2}$ defines an isomorphism $\scrH^{d\varsigma}(M,g)\to \scrL^2(M,g)$, \cref{eq:j923h} fails to be a Hilbert-Schmidt map into $\scrL^2(M)$ if and only if 
	\begin{equation} (1+\triangle_g)^{s/2} \operatorname{Op}(\chi_m) (1+\triangle_g)^{-d\varsigma/2}:\scrL^2(M,g) \to \scrD'(M)  \label{eq:j923j}
	\end{equation} 
	fails to be a Hilbert-Schmidt map into $\scrL^2(M,g)$. 
	By the commutativity of the principal symbol map, the pseudodifferential operator \cref{eq:j923j} is microlocally speaking somewhere elliptic of order $s-d\varsigma$, $\operatorname{Ell}^{s-d\varsigma}((1+\triangle_g)^{s/2} \operatorname{Op}(\chi_m) (1+\triangle_g)^{-d\varsigma/2}) \neq \varnothing$.

	We can now appeal to \Cref{HS_criterion} to conclude that the members of the preceding family of maps fail to be Hilbert-Schmidt maps into $\scrL^2(M,g)$ if and only if $s-d\varsigma \geq -d/2$, that is if and only if $s\geq d(\varsigma-1/2)$. 
\end{proof}

\begin{proof}[Proof of \Cref{thm:prelim}]
	We know that $\Gamma$ is the law of Gaussian noise with $\sigma_n = (1+\lambda_n)^{-1/2}$. By Weyl's law,  $\sigma_n$ obeys polynomial asymptotics with $\varsigma = 1/d$ in \cref{eq:zl}. \Cref{thm:main} therefore gives the desired result.
\end{proof}

\section*{Acknowledgments}

This research was partially supported by a Hertz fellowship. 

The problem of the microlocal (ir)regularity of random distributions sprung out of a somewhat related question asked by my advisor, Peter Hintz, during my qualifying exam,  and it goes without saying that I am indebted to his comments and guidance. 

In addition, I express my gratitude to the anonymous reviewer for their comments.

\printbibliography

\end{document}